\documentclass[a4paper]{scrartcl}

\usepackage[T1]{fontenc}
\usepackage[utf8]{inputenc}
\usepackage[english]{babel}
\usepackage{amsfonts}
\usepackage{amsmath}

\usepackage{amsthm}
\usepackage{paralist}
\usepackage{amssymb}
\usepackage{mathrsfs}
\usepackage{mathabx} 
\usepackage{comment} 
\usepackage{color} 
\usepackage{xcolor}
\usepackage[a4paper,top=3cm,bottom=4cm]{geometry}

\numberwithin{equation}{section} 

\newcommand{\norm}[1]{\| #1 \|}
\newcommand{\Norm}[1]{ \left\| #1 \right\| }

\newcommand{\Fourier}{\mathscr{F}}

\newcommand{\Eps}{\mathcal{E}}
\newcommand{\leer}{\emptyset}
\newcommand{\Punkt}{\thinspace \text{{\raisebox{-1.0pt}{\Large $\cdot$}}} \thinspace} 
\newcommand{\Supp}{\mathrm{spt}}
\newcommand{\Propa}{\text{property}\thinspace (\alpha)}

\newcommand{\Schwartz}{\mathscr{S}(\mathbb{R}^n)}
\newcommand{\Temp}{\mathscr{S}'(\mathbb{R}^n)}

\newcommand{\RR}{\mathcal{R}}
\newcommand{\Bild}{\mathscr{R}}
\newcommand{\Kern}{\mathscr{N}}
\newcommand{\Def}{\mathscr{D}}
\newcommand{\grad}{\mathscr{G}} 
\newcommand{\gradd}{\mathscr{G}^*} 
\newcommand{\A}{\mathscr{A}}
\newcommand{\HT}{\mathcal{HT}}
\newcommand{\timeD}{\frac{d}{dt}}

\excludecomment{AUSBLENDEN}

\newtheorem{satz}{Proposition}[section]
\newtheorem{theorem}[satz]{Theorem}
\newtheorem{lemma}[satz]{Lemma}
\newtheorem{korollar}[satz]{Corollary}
\theoremstyle{definition}
\newtheorem{definition}[satz]{Definition}

\newtheorem{bemerkung}[satz]{Remark}
\newtheorem{beispiel}[satz]{Example}

\begin{document}

\title{Triebel-Lizorkin-Lorentz spaces\\ and the Navier-Stokes equations}

\author{Pascal Hobus and J\"urgen Saal}

\date{September 26, 2017}

\maketitle

\begin{abstract}
We derive basic properties of Triebel-Lizorkin-Lorentz spaces important 
in the treatment of PDE. For instance, we prove Triebel-Lizorkin-Lorentz spaces 
to be of class $\HT$, to have property $(\alpha)$, and to admit a
multiplier result of Mikhlin type. By utilizing these properties
we prove the Laplace and the Stokes operator to admit a bounded
$H^\infty$-calculus. This is finally applied to derive
local strong well-posedness for the Navier-Stokes
equations on corresponding Triebel-Lizorkin-Lorentz ground spaces. 
\end{abstract}



\section{Introduction}

The Triebel-Lizorkin-Lorentz spaces $F^{s,r}_{p,q}$, a unification of
Triebel-Lizorkin spaces $F^s_{p,q}$ and Lorentz spaces $L_{p,r}$, were
introduced by Yang, Cheng and Peng (see
\cite{Yang-Cheng-Peng}) in 2005, where the possible parameters are $s
\in \mathbb{R}$, $1 < p,q < \infty$ and $1 \le r \le \infty$. 
Implicitly the spaces $F^{s,r}_{p,q}$ already appear in the pertinent
monograph of Triebel (see~\cite[Sec.\ 2.4.2]{Triebel1978}).
By means
of wavelet theory in \cite{Yang-Cheng-Peng} $F^{s,r}_{p,q}$ is proved to
be a real interpolation scale of Triebel-Lizorkin spaces, which is
very important and helpful, in particular for applications to PDE.
In 2011, Xiang and Yan already considered Triebel-Lizorkin-Lorentz
spaces in the context of partial differential equations
and established the local well-posedness of a quasi-geostrophic equation 
(see~\cite{Xiang-Yan}).

The scale $F^{s,r}_{p,q}$ contains many important function spaces: By
setting $r = p$, we obtain the Bessel-potential spaces $H^s_p$ for $q = 2$ as well as the Sobolev-Slobodeckij spaces $W^s_p$ for $q = p$ in the case $s \not \in \mathbb{Z}$ resp. $q = 2$ in the case $s \in \mathbb{Z}$.
In particular, we obtain the Lebesgue spaces $L_p$ by setting $s = 0$ as
well as the Lorentz spaces $L_{p,r} = F^{0,r}_{p,2}$.

It is therefore natural to ask, whether the scale of  
Triebel-Lizorkin-Lorentz spaces is suitable in the treatment of partial
differential equations, since a corresponding outcome would yield
results simultaneously in all spaces listed above.

The purpose of 
this paper is twofold. First we establish further fundamental properties
of Triebel-Lizorkin-Lorentz spaces, such as of class $\HT$, property
$(\alpha)$, useful equivalent norms, a Mikhlin type multiplier result,
etc.\ (see Section~\ref{secproptll}). 
Second, we apply these properties in order to prove a bounded
$H^\infty$-calculus for the Laplace and the Stokes operator 
(Section~\ref{secLap} and \ref{sec: Helmholtz Projection and Stokes Operator})
which, in turn, will then be utilized to construct a maximal strong solution
$(u,\nabla p)$ of the Navier-Stokes equations
\begin{equation*}
\mathrm{(NSE)}_{f,u_0}
\begin{cases}
\timeD u - \Delta u + \nabla p + (u \cdot \nabla)u &= ~f \quad ~\text{in } (0,T) \times \mathbb{R}^n, \\
\qquad \qquad \qquad \qquad \quad~ \mathrm{div} \thinspace u &= ~0 \quad ~\text{in } (0,T) \times \mathbb{R}^n, \\
 \qquad \qquad \qquad \qquad \quad~~ u(0) &= ~u_0 \quad \text{in } \mathbb{R}^n
\end{cases}
\end{equation*}
in these spaces (Section~\ref{sec: nse}).
In fact, we prove the following result.

\begin{theorem} \label{thm: Hauptsatz0}
Let $n \in \mathbb{N}$, $n \ge 2$, $s > -1$ and let $1 < p,q,r < \infty$
and $1 < \eta < \infty$ such that $\frac{n}{2p} + \frac{1}{\eta} < 1$.
Then for every
$f \in L_\eta((0,\infty),(F^{s,r}_{p,q})^n)$ and initial value 
$u_0 \in\bigl(F^{s,r}_{p,q},\, F^{s+2,r}_{p,q}\bigr)^n_{1-1/\eta,\eta}$ with
vanishing divergence, there is a maximal time $T^*>0$ such that 
the Navier-Stokes equations $\mathrm{(NSE)}_{f,u_0}$
have a unique maximal strong solution $(u,\nabla p)$ on $(0,T^*)$ satisfying
\begin{equation*}
\begin{split}
u &\in H^1_\eta \big( (0,T),(F^{s,r}_{p,q})^n \big) \cap L_\eta \big( (0,T), (F^{s+2,r}_{p,q})^n \big), \\
\nabla p &\in L_\eta((0,T),(F^{s,r}_{p,q})^n)
\end{split}
\end{equation*}
for every $T\in (0,T^*)$.
If additionally $\frac{n}{2p} + \frac{2}{\eta} < 1$, then $u$ is either a global solution or we have
\begin{equation} \label{eq: Blowup}
T^* < \infty \quad \text{and} \quad
\underset{t \nearrow T^*}{\mathrm{limsup}} \thinspace
\| u(t) \|_{\bigl(F^{s,r}_{p,q},\, F^{s+2,r}_{p,q}\bigr)^n_{1-1/\eta,\eta}} 
= \infty.
\end{equation}
\end{theorem}

\begin{bemerkung}
The constraints on the parameters $p,\eta$, especially the more 
restrictive one for the additional property that $u$ is either a global solution or \eqref{eq: Blowup} holds, rely on the use of 
the multiplication result for $F^{s,r}_{p,q}$-spaces given 
in Lemma~\ref{thm: Produkt}. They might be improved
to the standard contraints in classical function spaces such as $L^p$. 
This, however, requires optimal results on multiplication for 
$F^{s,r}_{p,q}$-spaces which by now are not available and would
go beyond the scope of this note.
\end{bemerkung}

In order to prove Theorem~\ref{thm: Hauptsatz0} we consider the 
usual operatorial formulation relying on the use of Helmholtz projection and
Stokes operator (see Theorem~\ref{thm: Hauptsatz}). 
Existence of the Helmholtz decomposition and 
maximal regularity for the Stokes operator in Triebel-Lizorkin-Lorentz spaces 
are proved in Section~\ref{sec: Helmholtz Projection and Stokes
Operator}. 
The concept of maximal regularity is introduced in the next section.
Starting from the pioneering works of Leray, Hopf, Fujita, Kato, 
Solonnikov, Giga, etc., 
local well-posedness for the Navier-Stokes equations 
in classical functions spaces
has been proved by a vast number of authors. We refrain from
trying to give a complete list here. Instead we refer to the
well-known monographs \cite{Sohr,Galdi} and the literature cited therein. 
For a comprehensive survey of results on the 
associated linear Stokes operator we also refer to \cite{hisa2016}.
As mentioned above, our approach has the advantage that it unifies
many of the existing results on local well-posedness by the fact
that Triebel-Lizorkin-Lorentz spaces include quite a number of
classical function spaces. The approach to the Navier-Stokes equations 
given here is meant as a first step towards a theory 
in Triebel-Lizorkin-Lorentz spaces.
Further investigations such as for instance well-posedness in critical
Triebel-Lizorkin-Lorentz spaces are left as a future challenge.

This article is organized as follows. In Section~\ref{sec: Basic}
we fix notation and recall briefly basic notions and tools related
to maximal regularity. In Section~\ref{secproptll} we establish
further properties of Triebel-Lizorkin-Lorentz spaces relevant
for the handling of PDE. In Section~\ref{secLap} we prove 
a bounded $H^\infty$-calculus for the Laplacian and in 
Section~\ref{sec: Helmholtz Projection and Stokes Operator}
existence of the Helmholtz decomposition and 
a bounded $H^\infty$-calculus for the Stokes operator on 
solenoidal subspaces.
The same property is proved to hold for the time derivative operator 
in Section~\ref{sec: time der} where we also give embedding results
important to handle nonlinearities. 
Theorem~\ref{thm: Hauptsatz0} is proved in Section~\ref{sec: nse}.
Finally, Appendix~\ref{sec: appendix} collects basic facts on
extension operators used in the sections before.

\section{Basic notation and preliminary results} \label{sec: Basic}

Generally $| \cdot |$ denotes the euklidean norm on $\mathbb{R}^n$ and the natural numbers $\mathbb{N}$ do not contain zero whereas $\mathbb{N}_0 = \mathbb{N} \cup \{ 0 \}$.
For two equivalent norms $\| \cdot \|$ and $\| \cdot \|'$ on a vector space $X$ we write $\| \cdot \| \sim \| \cdot \|'$.
Similarly, we use the notation $\| \cdot \| \lesssim \| \cdot \|'$ if there is a constant $C > 0$ such that $\| \cdot \| \le C \| \cdot \|'$.
Commonly $C$ denotes a generic positive constant. Especially during estimates we also use $C',C'',\dots$ when the constant changes.
The \emph{space of Schwartz functions} is denoted by $\Schwartz$ and thus $\Temp$ is the space of \emph{tempered distributions}. The corresponding space of $X$-valued Schwartz functions (where $X$ is a Banach space) is $\mathscr{S}(\mathbb{R}^n,X)$ and we set $\mathscr{S}'(\mathbb{R}^n,X) = \mathscr{L}(\Schwartz,X)$, that is, the space of continuous linear operators $T: \Schwartz \rightarrow X$.
The space of Distributions is $\mathscr{D}'(\mathbb{R}^n)$.
Nullspace resp.\ Range of a linear operator $T: \Def(T) \subset X \rightarrow X$ are denoted by $\Kern(T)$ resp.\ $\Bild(T)$.
The support of a function $f$ is denoted by $spt(f)$. For a Banach space $X$ and a measure space $(\Omega,\mathcal{A},\mu)$ let $\mathcal{M}(\Omega,X)$ be the space of measurable (i.e., also separable-valued) functions $f: \Omega \rightarrow X$.
The \emph{Lorentz space} $L_{p,r}(X) = L_{p,r}(\Omega,X) \subset \mathcal{M}(\Omega,X)$ with parameters $1 \le p,r \le \infty$ consists of those functions whose Lorentz quasinorm
\begin{equation*}
\vvvert f \vvvert_{L_{p,r}(\Omega,X)} =
\begin{cases}
\Big( \int_0^\infty \big[t^\frac{1}{p} f^*(t) \big]^r \frac{dt}{t} \Big)^\frac{1}{r}, & r < \infty \\
\sup_{t>0} t^\frac{1}{p} f^*(t), & r = \infty
\end{cases}
\end{equation*}
is finite, where
\begin{equation*}
f^*(t) = \inf \{ \alpha \ge 0 : d_f(\alpha) \le t \}, \quad t \ge 0
\end{equation*}
is the \emph{decreasing rearrangement} and
\begin{equation*}
d_f(\alpha) = \mu( \{ z \in \Omega : \| f(z) \| > \alpha \} ), \quad \alpha \ge 0
\end{equation*}
is the \emph{distribution function} of $f \in \mathcal{M}(\Omega,X)$. Two functions in $L_{p,r}(\Omega,X)$ are considered equal, if they are equal on a null set (with respect to $\mu$).

For $1 < p_0,p_1,p < \infty$, $p_0 \ne p_1$, $1 \le r_0,r_1,r \le \infty$ and $0 < \theta < 1$ such that $\frac{1}{p} = \frac{1 - \theta}{p_0} + \frac{\theta}{p_1}$ we have
\begin{equation*}
\big( L_{p_0,r_0}(X) , L_{p_1,r_1}(X) \big)_{\theta,r} = L_{p,r}(X)
\end{equation*}
(see \cite[Rem. 1.18.6/4]{Triebel1978})
where $(\cdot,\cdot)_{\theta,r}$ denotes the real interpolation functor. In view of the identity $L_{p,p}(X) = L_p(X)$ the Lorentz spaces are identified as real interpolation spaces of the Lebesgue spaces $L_p(X)$.
Note that $L_{p,r}(\Omega,X)$ is hence normable in the case $p > 1$. We denote the corresponding norm by $\| \cdot \|_{L_{p,r}(X)}$.

The space $l^s_q(X)$ (for $s \in \mathbb{R}$ and $1 \le q < \infty$) consists of the sequences $a = (a_k)_{k \in \mathbb{N}_0} \subset X$ in a Banach space $X$ that fulfill
\begin{equation*}
\| a \|_{l^s_q(X)} = \bigg(\sum_{k \in \mathbb{N}_0} \big[ 2^{-ks} \| a_k \| \big]^q \bigg)^\frac{1}{q} < \infty.
\end{equation*}
In the case $X = \mathbb{C}$ we write $l^s_q$ instead of $l^s_q(\mathbb{C})$.

Real resp.\ complex interpolation of the Sobolev spaces
\begin{equation*}
W^k_p(\mathbb{R}^n,X) = \big\{ u \in L_p(\mathbb{R}^n,X) : \partial^\alpha u \in L_p(\mathbb{R}^n,X) ~\forall \alpha \in \mathbb{N}_0^n, |\alpha| \le k \big\}
\end{equation*}
leads to \emph{Sobolev-Slobodeckij spaces}
\begin{equation*}
W^s_p(\mathbb{R}^n,X) = \big( L_p(\mathbb{R}^n,X),W^k_p(\mathbb{R}^n,X) \big)_{\frac{s}{k},p}
\end{equation*}
resp.\ \emph{Bessel-potential spaces}
\begin{equation*}
H^s_p(\mathbb{R}^n,X) = \big[ L_p(\mathbb{R}^n,X),W^k_p(\mathbb{R}^n,X) \big]_{\frac{s}{k}},
\end{equation*}
where $k \in \mathbb{N}$, $1 < p < \infty$, $0 < s < k$ and $X$ is a complex Banach space. In the following we assume that $X$ is of class $\HT$ (we give one possible definition for spaces of class $\HT$ below). In the case $s = m \in \mathbb{N}_0$ the Bessel-potential spaces are given by the Sobolev spaces, so $H^m_p(\mathbb{R}^n,X) = W^m_p(\mathbb{R}^n,X)$.
For $s \in \mathbb{R}$ and $1 < p < \infty$ we will also use the representation
\begin{equation*}
H^s_p(\mathbb{R}^n,X) = \big\{ u \in \mathscr{S}'(\mathbb{R}^n,X) : \Fourier^{-1} (1 + |\xi|^2)^\frac{s}{2} \Fourier u \in L_p(\mathbb{R}^n,X) \big\},
\end{equation*}
where $\| \Fourier^{-1} (1 + |\xi|^2)^\frac{s}{2} \Fourier \cdot \|_{L_p(\mathbb{R}^n,X)}$ is an equivalent norm in $H^s_p(\mathbb{R}^n,X)$ and $\Fourier$ denotes the Fourier transform.
Moreover, the continuous embeddings
\begin{equation*}
H^s_p(\mathbb{R}^n,X)
\subset W^{s-\epsilon}_p(\mathbb{R}^n,X)
\subset H^{s-2 \epsilon}_p(\mathbb{R}^n,X)
\end{equation*}
hold for any $\epsilon > 0$.
We refer to~\cite{Hytoenen} and~\cite{Amann97} for a detailed treatise of Bessel-potential and Sobolev-Slobodeckij spaces.
In general, if $\mathcal{F}(\mathbb{R}^n,X)$ is any normed function space (e.g. $\mathcal{F} = H^s_p$ or $\mathcal{F} = W^s_p$) and $\Omega \subset \mathbb{R}^n$ is any domain, we denote by $\mathcal{F}(\Omega,X)$ the space of restrictions of functions $u \in \mathcal{F}(\mathbb{R}^n,X)$ to $\Omega$, equipped with the norm $\| u \|_{\mathcal{F}(\Omega,X)} = \inf \{ \| v \|_{\mathcal{F}(\mathbb{R}^n,X)} : v \in \mathcal{F}(\mathbb{R}^n,X), v|_\Omega = u \}$.

In order to deal with operator-valued Fourier multipliers we employ the following concept.
Let $X,Y$ be complex Banach spaces. Let $\mathcal{E}_P$ denote the set of families of random variables $(\epsilon_i)_{i \in I}$ on a probability space $P = (\Omega,\mathcal{A},\mu)$ with values in $\{ \pm 1 \}$, which are independent and symmetrically distributed. We refer to a familiy of continuous linear operators $\mathcal{T} \subset \mathscr{L}(X,Y)$ as \emph{$\RR$-bounded} if there is a probability space $P = (\Omega,\mathcal{A},\mu)$ with $\mathcal{E}_P \ne \emptyset$, $p \in [1,\infty)$ and a constant $C > 0$ such that for all $N \in \mathbb{N}$, $(\epsilon_1,\dots,\epsilon_N) \in \mathcal{E}_P$, $T_i \in \mathcal{T}$ and $x_i \in X$ (for $1 \le i \le N$)
\begin{equation} \label{eq: R-beschraenkt}
\left\| \sum_{i=1}^N \epsilon_i T_i x_i \right\|_{L_p(\Omega,Y)}
\le C \left\| \sum_{i=1}^N \epsilon_i x_i \right\|_{L_p(\Omega,X)}.
\end{equation}
In this case we call $\RR_p(\mathcal{T}) := \inf \{ C > 0 : \eqref{eq: R-beschraenkt} \text{ holds} \}$ the \emph{$\RR$-bound} or the \emph{$\RR_p$-bound}.
Note that $\RR$-boundedness implies the boundedness of $\mathcal{T} \subset \mathscr{L}(X,Y)$. If a family $\mathcal{T} \subset \mathscr{L}(X,Y)$ is $\RR_p$-bounded for a $p \in [1,\infty)$ then it is also $\RR_q$-bounded for any $q \in (1,\infty)$. In this case \eqref{eq: R-beschraenkt} also holds for a (possibly different) constant $C > 0$ if we replace $P$ by an arbitrary probability space $Q$ with $\mathcal{E}_Q \ne \emptyset$.
Also note that, in view of Lebesgue's dominated convergence theorem, it is sufficient to claim~\eqref{eq: R-beschraenkt} for $x_i$ in a dense subspace of $X$.
The following result is useful to extend boundedness to $\RR$-boundedness in some concrete cases (see \cite[Lemma 3.5]{Denk-Hieber-Pruess}).

\begin{theorem}[Kahane's contraction principle]
Let $X$ be a Banach space over $\mathbb{F} \in \{ \mathbb{R},\mathbb{C} \}$, $P = (\Omega,\mathcal{A},\mu)$ a probability space and $1 \le p < \infty$. Let $N \in \mathbb{N}$ and $a_j,b_j \in \mathbb{K}$ with $|a_j| \le |b_j|$ for $j = 1,\dots,N$. Then we have for all $x_1,\dots,x_N \in X$ and $\epsilon_1,\dots,\epsilon_N \in \Eps_P$
\begin{equation*}
\Big\| \sum_{i=1}^N a_i \epsilon_i x_i \Big\|_{L_p(\Omega,X)}
\le C_\mathbb{F} \Big\| \sum_{i=1}^N b_i \epsilon_i x_i \Big\|_{L_p(\Omega,X)},
\end{equation*}
where $C_\mathbb{R} = 1$ and $C_\mathbb{C} = 2$.
\end{theorem}

We call a linear and densely defined operator $A: \Def(A) \subset X
\rightarrow X$ \emph{pseudo-sectorial} if its spectrum $\sigma(A)$ is
contained in a closed sector $\overline{\Sigma}_\varphi$ with angle
$\varphi \in (0,\pi)$, where $\Sigma_\varphi = \{ z \in \mathbb{C}
\setminus \{ 0 \} : |\arg(z)| < \varphi \}$, and the family $\{ \lambda
(\lambda + A)^{-1} : \lambda \in \Sigma_{\pi - \varphi} \} \subset
\mathscr{L}(X)$ is bounded. If $\{ \lambda (\lambda + A)^{-1} : \lambda
\in \Sigma_{\pi - \varphi} \} \subset \mathscr{L}(X)$ is even
$\RR$-bounded, $A$ is called \emph{pseudo-$\RR$-sectorial}. We omit the prefix ''pseudo-'' if the range $\Bild(A) \subset X$ is dense and so we get a \emph{sectorial} resp.\ \emph{$\RR$-sectorial} operator.
We denote the infimum over all $\varphi \in (0,\pi)$, such that the family $\{ \lambda (\lambda + A)^{-1} : \lambda \in \Sigma_{\pi - \varphi} \} \subset \mathscr{L}(X)$ is bounded, by $\varphi_A$ (\emph{spectral angle}) if $A$ is a (pseudo-)sectorial operator and likewise $\varphi^\RR_A$ is the infimum over all $\varphi \in (0,\pi)$ such that this family is $\RR$-bounded if $A$ is a (pseudo-)$\RR$-sectorial operator.

For a pseudo-sectorial operator $A$ and a fixed angle $\varphi < \varphi_A$ we will make use of the Dunford calculus
\begin{equation*}
f \longmapsto f(A),
\end{equation*}
which maps a function $f \in \mathscr{H}_0(\Sigma_\varphi) = \bigcup_{\alpha,\beta < 0} \mathscr{H}_{\alpha,\beta}(\Sigma_\varphi)$ to a bounded operator on $X$,
as well as of its extension to $\mathscr{H}_p(\Sigma_\varphi) = \bigcup_{\alpha \in \mathbb{R}} \mathscr{H}_{\alpha,\alpha}(\Sigma_\varphi)$, when $A$ is sectorial.
Here $\mathscr{H}_{\alpha,\beta}(\Sigma_\varphi)$ is the space of holomorphic functions $f: \Sigma_\varphi \rightarrow \mathbb{C}$ such that
\begin{equation*}
\| f \|_{\varphi,\alpha,\beta}
= \sup_{|z| \le 1} |z^\alpha f(z)| + \sup_{|z| > 1} |z^{-\beta} f(z)|
\end{equation*}
is finite.
We refer to~\cite{Denk-Hieber-Pruess} for a precise definition and treatise of this functional calculus. Note that for a function $f \in \mathscr{H}_{\alpha,\alpha}(\Sigma_\varphi)$ we get a bounded operator $f(A)$ in the case $\alpha < 0$ but in general only a closed operator on the domain
$\Def(f(A)) = \{ x \in X : (g^k f)(A)x \in \Def(A^k) \cap \Bild(A^k) \}$.
Here $k > \alpha$ is a nonnegative integer and $g \in \mathscr{H}_0(\Sigma_\varphi)$ is the function $g(z) := \frac{z}{(1 + z)^2}$, which leads to a bijective mapping
$g(A)^{-k}: \Def(A^k) \cap \Bild(A^k) \longrightarrow X$.
A slight modification of this functional calculus leads to the following well-known characterization of sectorial operators
(see~\cite[Prop. 3.4.4]{Haase}).

\begin{satz}
An operator
$A: \Def(A) \subset X \rightarrow X$ is pseudo-sectorial with angle $\varphi_A < \frac{\pi}{2}$ if and only if $-A$ is the generator of a bounded holomorphic strongly continuous semigroup.
\end{satz}

The just introduced functional calculus is also used to describe an important property of an operator $A$. Let
$A: \Def(A) \subset X \longrightarrow X$
be a sectorial operator. Then $A$
has a \emph{bounded $H^\infty$-calculus} in $X$ if for some $\varphi \in (\varphi_A, \pi)$ there is a constant $C_\varphi > 0$ such that for any $f \in \mathscr{H}_0(\Sigma_\varphi)$ we have
\begin{equation} \label{eq: H^infty}
\| f(A) \|_{\mathscr{L}(X)} \le C_\varphi \| f \|_{L_\infty(\Sigma_\varphi)}.
\end{equation}
In this case $\eqref{eq: H^infty}$ also holds for all bounded holomorphic functions $f$ on $\Sigma_\varphi$.
The infimum over all angles $\varphi \in (\varphi_A, \pi)$, such that $\eqref{eq: H^infty}$ holds with a constant $C_\varphi > 0$, is called \emph{$H^\infty$-angle} and is denoted by $\varphi_A^\infty$.
Likewise we say that $A$ has an \emph{$\RR$-bounded $H^\infty$-calculus} in $X$ if the set
\begin{equation*}
\{ f(A) : f \in \mathscr{H}_0(\Sigma_\varphi), \| f \|_{L_\infty(\Sigma_\varphi)} \le 1 \} \subset \mathscr{L}(X)
\end{equation*}
is $\RR$-bounded and the related $\RR H^\infty$-angle is denoted by $\varphi_A^{\RR,\infty}$.
If $A$ has a bounded $H^\infty$-calculus, then
\begin{equation} \label{eq: H^infty und komplexe Interpolation}
\Def(A^\alpha) = [X,\Def(A)]_\alpha
\end{equation}
holds for all $0 < \alpha < 1$ (see \cite{Denk-Hieber-Pruess}), where the fractional power $A^\alpha: \Def(A^\alpha) \subset X \longrightarrow X$ is defined via the functional calculus above with the function $z \mapsto z^\alpha$.

We recall the following two assertions that frequently occur 
in the context of the $H^\infty$-calculus. They have been proved in~\cite{Denk-Saal-Seiler} (see the proofs of Prop.~2.9 and Prop.~2.7, respectively).

\begin{lemma} \label{thm: Symbolabschaetzung}
~
\begin{enumerate}[(a)]
\item \label{thm: Symbolabschaetzung 1} Let $0 < \varphi < \pi$. Then for all $\alpha \in \mathbb{N}_0^n$ there is a constant $C_{\alpha,\varphi} > 0$ so that for every holomorphic and bounded function $h: \Sigma_\varphi \rightarrow \mathbb{C}$ we have
\begin{equation*}
\sup_{\xi \in \mathbb{R}^n \setminus \{ 0 \}} |\xi|^{|\alpha|} | \partial^\alpha h(|\xi|^2) | \le C_{\alpha,\varphi} \| h \|_{L_\infty(\Sigma_\varphi)}.
\end{equation*}
\item \label{thm: Symbolabschaetzung 2} Let $\frac{\pi}{2} < \varphi < \pi$. Then for $k=0,1$ there is a constant $C_{\varphi} > 0$ so that for every $h \in \mathscr{H}_0(\Sigma_\varphi)$ we have
\begin{equation*}
\sup_{\xi \in \mathbb{R} \setminus \{ 0 \}} |\xi|^k | \partial^k h(i \xi) | \le C_{\varphi} \| h \|_{L_\infty(\Sigma_\varphi)}.
\end{equation*}
\end{enumerate}
\end{lemma}

Next we give the definition of maximal regularity (due to \cite{Kunstmann-Weis}) for an operator $A: \Def(A) \subset X \rightarrow X$, which is the generator of a bounded holomorphic strongly continuous semigroup $(S(t))_{t \ge 0}$ on a complex Banach space $X$. Therefore, we fix $1 < p < \infty$ and $0 < T \le \infty$. $A$ has \emph{maximal $L_p$-regularity} on $(0,T)$ if for all $f \in L_p((0,T),X)$ the solution
\begin{equation*}
u(t) = \int_0^t S(t-s) f(s) ds
\end{equation*}
of the Cauchy problem
\begin{equation}
\begin{cases}
u'(t) - Au(t) &= f(t), \quad t \in (0,T) \\
\qquad \qquad u(0) &= 0
\end{cases}
\end{equation}
is Fréchet differentiable a.e., takes its values in $\Def(A)$ a.e.\ and $u',Au \in L_p((0,T),X)$. In this case we get
\begin{equation} \label{eq: Maximale Regularität}
\| u' \|_{L_p((0,T),X)} + \| Au \|_{L_p((0,T),X)} \le C \| f \|_{L_p((0,T),X)}
\end{equation}
by application of the closed graph theorem.
We write $A \in \mathrm{MR}(X,C)$ if $A$ has maximal $L_p$-regularity for some (or equivalently for all) $1 < p < \infty$ on some $(0,T)$ so that~\eqref{eq: Maximale Regularität} holds with a constant $C = C(T) > 0$. If $A \in \mathrm{MR}(X,C)$ and $C$ doesn't depend on $T$ (i.e.,~\eqref{eq: Maximale Regularität} holds for $T = \infty$) we write $A \in \mathrm{MR}(X)$.

Now we take a look at the advantages of maximal regularity. Again for a complex Banach space $X$, let
$A: \Def(A) \subset X \rightarrow X$ be the generator of a bounded holomorphic strongly continuous semigroup. For $1 < p < \infty$ and $T \in (0,\infty]$ we set
\begin{equation*}
\mathbb{E}_T := H^1_p((0,T),X) \cap L_p((0,T),\Def(A))
\end{equation*}
(the solution space of the related Cauchy problem)
and
\begin{equation*}
\mathbb{F}_T \times \mathbb{I}
:= L_p((0,T),X) \times \big\{ x = u(0) : u \in \mathbb{E}_T \big\}
\end{equation*}
(the data space).
Note that $\mathbb{I}$ is a Banach space with the norm $\| x \|_\mathbb{I} = \inf_{u(0) = x} \| u \|_{\mathbb{E}_T}$, independent of $T$ and we have 
\begin{equation} \label{charini}
\mathbb{I} = \big( X,\Def(A) \big)_{1 - \frac{1}{p},p}
\end{equation}
(see~\cite{Pruess-Simonett}, Prop. 3.4.4). If
$A$ has maximal $L_p$-regularity on a finite interval $(0,T)$, then the solution operator
\begin{equation} \label{eq: Lösungsoperator}
L: \mathbb{E}_T \longrightarrow \mathbb{F}_T \times \mathbb{I}, \quad
u \longmapsto
\begin{pmatrix}
(\timeD - A)u \\
u(0)
\end{pmatrix}
\end{equation}
is an isomorphism. This leads to the estimate
\begin{equation} \label{eq: Abschätzung MaxReg}
\| u \|_{H^1_p((0,T),X) \cap L_p((0,T),\Def(A))} \le C(T) \big( \| f \|_{L_p((0,T),X)} + \| x \|_\mathbb{I} \big),
\end{equation}
when $u := L^{-1} \left( \begin{smallmatrix} f \\ x \end{smallmatrix} \right)$ is the solution for some $\big( \begin{smallmatrix} f \\ x \end{smallmatrix} \big) \in \mathbb{F}_T \times \mathbb{I}$.

\begin{lemma} \label{thm: Abschätzung MR}
Let $X$ be a Banach space, $1 < p < \infty$, $0 < T_0 < \infty$ and let $A \in \mathrm{MR}(X,C(T_0))$. Then there exists a constant $C' = C'(T_0) > 0$ such that
\begin{equation*}
\| L^{-1} \left( \begin{smallmatrix} f \\ 0 \end{smallmatrix} \right) \|_{H^1_p((0,T),X) \cap L_p((0,T),\Def(A))} \le C'(T_0) \| f \|_{L_p((0,T),X)}
\end{equation*}
holds for all $T \in (0,T_0]$ and for all $f \in L_p((0,T),X)$,
where $L$ is the solution operator from~\eqref{eq: Lösungsoperator}.
\end{lemma}

\begin{proof}
By the trivial extension of $f \in L_p((0,T),X)$ to $(0,T_0)$ we get the estimate~\eqref{eq: Maximale Regularität} with a constant independent of $T \in (0,T_0]$. Now the assertion follows from the fact that the Poincaré inequality $\| u \|_{L_p((0,T),X)} \le K \| u' \|_{L_p((0,T),X)}$ holds with a constant $K > 0$, which is independent of $T \in (0,T_0]$ as well.
\end{proof}

\begin{lemma} \label{thm: Einbettung Phasenraum}
Let $1 < p < \infty$ and $T \in (0,\infty]$. Then, with the notation above, we have the continuous embedding
\begin{equation*}
\mathbb{E}_T \subset BUC([0,T),\mathbb{I})
\end{equation*}
(where $BUC$, as usual, means bounded and uniformly continuous).
Here the operator $A: \Def(A) \subset X \rightarrow X$ only needs to be closed and densely defined in a Banach space $X$.
\end{lemma}

\begin{proof}
The case $T = \infty$ follows essentially from the strong continuity of the translation semigroup. Then, by a standard extension and retraction argument one gets the case $T < \infty$ as a consequence. See~\cite[Prop. 1.4.2]{Amann_Quasilinear} for details.
\end{proof}

In the theory of partial differential equations, the notions of class $\HT$ and $\Propa$ for Banach spaces turned out to be significant.
A Banach space $X$ is of class $\mathcal{HT}$ if the \emph{Hilbert transform}
\begin{equation*}
H: \mathscr{S}(\mathbb{R},X) \longrightarrow \mathcal{M}(\mathbb{R},X), \quad
Hf(t) = \lim_{\epsilon \searrow 0} \int_{|s| > \epsilon} \frac{f(t-s)}{s} ds
\end{equation*}
has an extension $H \in \mathscr{L}(L_p(\mathbb{R},X))$ for any (or equivalently for one) $1 < p < \infty$.
A complex Banach space $X$ has \emph{$\Propa$} if there exist $1 \le p < \infty$, two probability spaces $P = (\Omega,\mathcal{A},\mu),P' = (\Omega',\mathcal{A}',\mu')$ with $\Eps_P,\Eps_{P'} \ne \leer$ and a constant $\alpha > 0$ such that for all $N \in \mathbb{N}$, $x_{ij} \in X$, $a_{ij} \in \mathbb{C}$, $|a_{ij}| \le 1 ~(i,j = 1, \dots ,N)$
and for all $(\epsilon_1, \dots ,\epsilon_N) \in \Eps_P$, $(\epsilon'_1, \dots ,\epsilon'_N) \in \Eps_{P'}$ we have
\begin{equation} \label{eq: Eigenschaft alpha}
\Norm{\sum_{i,j=1}^N \epsilon_i \epsilon'_j a_{ij} x_{ij}}_{L_p(\Omega \times \Omega',X)}
\le \alpha \Norm{\sum_{i,j=1}^N \epsilon_i \epsilon'_j x_{ij}}_{L_p(\Omega \times \Omega',X)}.
\end{equation}
A useful application of $\Propa$ is the following one,
which is a direct consequence of the Kalton-Weis theorem (see~\cite[Thm. 4.5.6]{Pruess-Simonett}).

\begin{theorem} \label{thm: H^infty und Propa impliziert RH^infty}
Let $X$ be a Banach space with $\Propa$ and let $A: D(A) \subset X \rightarrow X$ be an operator with a bounded $H^\infty$-calculus. Then $A$ has an $\RR$-bounded $H^\infty$-calculus with $\varphi_A^{\infty} = \varphi_A^{\RR,\infty}$.
\end{theorem}

The following operator-valued version of Mihklin's theorem is important for our purposes as well as the subsequent characterization of maximal $L_p$-regularity.
The results are due to Girardi and Weis (see~\cite{Weis-Girardi} or~\cite[Thm. 4.3.9, Thm. 4.4.4]{Pruess-Simonett}).

\begin{theorem} \label{thm: Satz von Mikhlin operatorwertig}
Let $X,Y$ be complex Banach spaces of class $\HT$, which have $\Propa$ and let $1 < p < \infty$.
For $m_\lambda \in C^n(\mathbb{R}^n \setminus \{ 0 \} , \mathscr{L}(X,Y))$, $~\lambda \in \Lambda$ assume that
$\kappa_\alpha
:= \RR_p \{ \xi^\alpha \partial^\alpha m_\lambda(\xi) :
                 \xi \in \mathbb{R}^n \setminus \{ 0 \}, \lambda \in \Lambda \}
< \infty$
for each $\alpha \in \{ 0,1 \}^n$.
Then the operator
\begin{equation*}
\Fourier^{-1} m_\lambda \Fourier : \mathscr{S}(\mathbb{R}^n,X) \longrightarrow \mathscr{S}'(\mathbb{R}^n,Y)
\end{equation*}
has a unique extension
$T_\lambda \in \mathscr{L}(L_p(\mathbb{R}^n,X),L_p(\mathbb{R}^n,Y))$
for every $\lambda \in \Lambda$ and we have
\begin{equation*}
\RR_p \{ T_\lambda : \lambda \in \Lambda \}
\le C_{p,n} \sum_{\alpha \in \{ 0,1 \}^n} \kappa_\alpha
=: C.
\end{equation*}
In particular, we have $\| \Fourier^{-1} m_\lambda \Fourier f \|_{L_p(Y)} \le C \| f \|_{L_p(X)}$ for
$f \in \mathscr{S}(\mathbb{R}^n,X)$ and $\lambda \in \Lambda$.
\end{theorem}

\begin{theorem}
Let $X$ be a Banach space of class $\HT$, $1 < p < \infty$ and let $-A: \Def(A) \subset X \rightarrow X$ be the generator of a bounded holomorphic strongly continuous semigroup. Then the following conditions are equivalent.
\begin{enumerate}[(i)]
\item $A$ has maximal $L_p$-regularity on $(0,\infty)$.
\item $A$ is pseudo-$\RR$-sectorial with $\varphi^\RR_A < \frac{\pi}{2}$.
\end{enumerate}
\end{theorem}

\section{Properties of Triebel-Lizorkin-Lorentz spaces}
\label{secproptll}

For parameters $s \in \mathbb{R}$, $1 < p,q < \infty$, $1 \le r \le \infty$ we call
\begin{equation*}
F^{s,r}_{p,q} := \big\{ u \in \Temp : \| u \|_{F^{s,r}_{p,q}} < \infty \big\}
\end{equation*}
the \emph{Triebel-Lizorkin-Lorentz space} (as defined in~\cite{Yang-Cheng-Peng}). The norm is given by
\begin{equation*}
\| u \|_{F^{s,r}_{p,q}} := \| (\varphi_k * u)_{k \in \mathbb{N}_0} \|_{L_{p,r}(l^s_q)}
\end{equation*}
where $(\widehat{\varphi}_k)_{k \in \mathbb{N}_0}$ is a dyadic decomposition defined as follows (cf.~\cite[Def.\ 2.3.1/2]{Triebel1978}).

\begin{definition}
Let $\boldsymbol{\Phi_N}$ (for $N \in \mathbb{N}$) denote the set of systems of functions $(\varphi_k)_{k \in \mathbb{N}_0} \subset \Schwartz$ with the following properties.
\begin{itemize}
\item $\widehat{\varphi}_k \ge 0$ for all $k \in \mathbb{N}_0$.
\item $\Supp(\widehat{\varphi}_k) \subset \{ 2^{k-N} \le | x | \le 2^{k+N} \}$
for $k \in \mathbb{N}$ and $\Supp(\widehat{\varphi}_0) \subset \{ | x | \le 2^N \}$.
\item There exist $D_1,D_2 > 0$ such that for all $\xi \in \mathbb{R}^n$
\begin{equation} \label{eq: Beschränktheit der Summe}
D_1 \le \sum_{k=0}^\infty \widehat{\varphi}_k(\xi) \le D_2.
\end{equation}
\item For any $\alpha \in \mathbb{N}_0^n$, there is $C_\alpha > 0$ such that for all $k \in \mathbb{N}_0$ and $\xi \in \mathbb{R}^n$
\begin{equation} \label{eq: Mikhlin-Bed fuer dyad. Zerl.}
|\xi|^{|\alpha|} |\partial^\alpha \widehat{\varphi}_k(\xi)| \le C_\alpha.
\end{equation}
\end{itemize}
Additionally, we set $\boldsymbol{\Phi} := \bigcup_{N \in \mathbb{N}} \boldsymbol{\Phi}_N$ and call each family $(\widehat{\varphi}_k)_{k \in \mathbb{N}_0}$ with $(\varphi_k)_{k \in \mathbb{N}_0} \in \boldsymbol{\Phi}$ a \emph{dyadic decomposition}.
\end{definition}


Note that the constant $C_\alpha$ in~\eqref{eq: Mikhlin-Bed fuer dyad. Zerl.} doesn't depend on the index $k$ but on the selected $N \in \mathbb{N}$, that is, on the radius $2^N$ of the dyadic decomposition. Also note that the existence of $D_2$ in~\eqref{eq: Beschränktheit der Summe} can be deduced from~\eqref{eq: Mikhlin-Bed fuer dyad. Zerl.} with $\alpha = 0$ and the properties of $\Supp(\widehat{\varphi}_k)$.
We will often use the following more specific dyadic decomposition.

\begin{beispiel} \label{thm: Bsp dyadische Zerlegung}
Let $\phi \in C^\infty(\mathbb{R}^n)$ be radial symmetric with $\Supp(\phi) \subset \{ |x| \le 1 \}$, $\phi = 1$ on $\{ |x| \le \frac{1}{2} \}$ and $0 \le \phi \le 1$. We set $\widehat{\psi}(\xi) := \phi(\frac{\xi}{2}) - \phi(\xi)$ and $\widehat{\psi}_k(\xi) := \widehat{\psi}(2^{-k} \xi)$ for $\xi \in \mathbb{R}^n$ and $k \in \mathbb{Z}$. Now we set $\varphi_k := \psi_k$ for $k \ge 1$ and define $\varphi_0 \in \Schwartz$ by
\begin{equation*}
\widehat{\varphi}_0(\xi) =
\begin{cases}
\sum_{j \le 0} \widehat{\psi}_j(\xi), & \text{if } \xi \ne 0 \\
1, & \text{if } \xi = 0.
\end{cases}
\end{equation*}
Then we get $(\varphi_k)_{k \in \mathbb{N}_0} \in \boldsymbol{\Phi}_1$ with $\sum_{k \in \mathbb{N}_0} \varphi_k(\xi) = 1$ for all $\xi \in \mathbb{R}^n$. In addition, we have the following (easy to verify) properties:
\begin{enumerate}[(a)]
\item \label{thm: Bsp dyadische Zerlegung 1} $\norm{\varphi_k}_{L_1} = \norm{\psi}_{L_1}$ for all $k \in \mathbb{N}$.
\item \label{thm: Bsp dyadische Zerlegung 2} $\sum_{k=1}^N \widehat{\varphi}_k \xrightarrow{N \to \infty} 1$ locally uniformly on $\mathbb{R}^n$.
\item \label{thm: Bsp dyadische Zerlegung 3} $\sum_{j=0}^N \varphi_j * f \xrightarrow{N \to \infty} f$ in $\Schwartz$ for all $f \in \Schwartz$.
\item \label{thm: Bsp dyadische Zerlegung 4} $\sum_{j=0}^N \varphi_j * u \xrightarrow{N \to \infty} u$ in $\Temp$ for all $u \in \Temp$.
\end{enumerate}
\end{beispiel}

If we replace the Lorentz-norm $\| \cdot \|_{L_{p,r}(l^s_q)}$ by $\| \cdot \|_{L_p(l^s_q)}$ then we get the well-known Triebel-Lizorkin spaces $F^s_{p,q}$. More precisely we have $F^{s,p}_{p,q} = F^s_{p,q}$.
One can find the following result as Remark 2.4.2/1 in~\cite{Triebel1978}.

\begin{satz}
The Triebel-Lizorkin-Lorentz spaces are independent of the choice of the dyadic decomposition.
\end{satz}

The following result is due to Yang, Cheng 
and Peng \cite{Yang-Cheng-Peng}. Their proof is based on wavelet
theory. We notice that it is possible to derive the following
interpolation property by $L^p$-interpolation and 
retraction and coretraction techniques
as developed in \cite{Triebel1978}, as well.

\begin{theorem} \label{thm: Interpolation von TLL-Räumen}
For $s \in \mathbb{R}$, $1 < p_0, p_1, q < \infty$, $1 \le r_0, r_1, r \le \infty$,
$p_0 \ne p_1$ and $0 < \theta < 1$ such that $\frac{1}{p} = \frac{1-\theta}{p_0} + \frac{\theta}{p_1}$
we have
\begin{equation*}
\big( F^{s,r_0}_{p_0,q} , F^{s,r_1}_{p_1,q} \big)_{\theta,r} = F^{s,r}_{p,q}.
\end{equation*}
\end{theorem}

\begin{lemma} \label{thm: Einbettungen}
Let $s \in \mathbb{R}$, $1 < p,q < \infty$ and $1 \le r \le \infty$. Then we have the following continuous embeddings.
\begin{enumerate}[(i)]
\item \label{thm: Einbettungen 1} $\Schwartz \subset F^{s,r}_{p,q} \subset \Temp$
and the first embedding is dense.
\item \label{thm: Einbettungen 2} $F^{s + \tau,r}_{p,q} \subset F^{s,r}_{p,q}$ for $\tau \ge 0$.
\item \label{thm: Einbettungen 3} $F^{s,r}_{p,q} \subset L_{p,r}$ if $s > 0$.
\end{enumerate}
\end{lemma}

\begin{proof}
\eqref{thm: Einbettungen 1} follows from the corresponding fact for Triebel-Lizorkin spaces since
\begin{equation*}
\Schwartz \subset F^s_{p_0,q} \cap F^s_{p_1,q} \subset F^{s,r}_{p,q} \subset F^s_{p_0,q} + F^s_{p_1,q} \subset \Temp
\end{equation*}
and since the intersection of an interpolation couple of Banach spaces is dense in their real interpolation space. \eqref{thm: Einbettungen 2} is a consequence of $l^{s + \tau}_q \subset l^s_q$. 

In order to prove~\eqref{thm: Einbettungen 3} we use the dyadic decomposition $(\widehat{\varphi}_k)_{k \in \mathbb{N}_0}$ of Example~\ref{thm: Bsp dyadische Zerlegung}. First we consider the estimate
\begin{equation} \label{eq: Abschätzung Reihe}
\Big\| \sum_{k=0}^\infty |\varphi_k * u| \Big\|_{L_{p,r}}
\le C \Big\| \Big( \frac{1}{2^{sk}} \Big)_{k \in \mathbb{N}_0} \Big\|_{l_{q'}} \| u \|_{F^{s,r}_{p,q}}
\end{equation}
that we get from Hölder's inequality with $\frac{1}{q'} + \frac{1}{q} = 1$. Since $s > 0$, the right-hand side is finite for $u \in F^{s,r}_{p,q}$. Applying Example \ref{thm: Bsp dyadische Zerlegung}~\eqref{thm: Bsp dyadische Zerlegung 4} we have $u = \sum_{k=0}^\infty \varphi_k * u$ where the convergence is in $\Temp$. Now~\eqref{eq: Abschätzung Reihe} gives that the series even converges pointwise a.e.\ and thus $u$ is a measurable function. On the other hand~\eqref{eq: Abschätzung Reihe} gives $\| u \|_{L_{p,r}} \le C' \| u \|_{F^{s,r}_{p,q}}$.
\end{proof}

\begin{satz} \label{thm: F^(s,r)_(p,q) UMD-Raum}
$F^{s,r}_{p,q}$ is of class $\HT$ for $s \in \mathbb{R}$ and $1 < p,q,r < \infty$.
\end{satz}

\begin{proof}
We need to show that the Hilbert transform
\begin{equation*}
H: \mathscr{S}(\mathbb{R},F^{s,r}_{p,q}) \longrightarrow \mathcal{M}(\mathbb{R},F^{s,r}_{p,q}), \quad
Hf(t) = \lim_{\epsilon \searrow 0} \int_{|s| > \epsilon} \frac{f(t-s)}{s} ds
\end{equation*}
has an extension $H \in \mathscr{L}(L_p(\mathbb{R},F^{s,r}_{p,q}))$.
For any $s \in \mathbb{R}$ and $1 < q < \infty$, Tonelli's theorem implies that $L_q(\mathbb{R},l^s_q)$ is a space of class $\HT$ and so is $L_p(\mathbb{R},l^s_q)$ for arbitrary $1 < p < \infty$. Since the Triebel-Lizorkin space $F^s_{p,q}$ is a retrakt of $L_p(\mathbb{R},l^s_q)$ we can transfer the $\HT$-Property to $F^s_{p,q}$ for any $s \in \mathbb{R}$ and $1 < p,q < \infty$.

Now for fixed parameters $s,p,q,r$ as in the assertion we can use Theorem~\ref{thm: Interpolation von TLL-Räumen} to complete the proof. As a direct conclusion of the interpolation property
$L_r(\mathbb{R},(X_0,X_1)_{\theta,r}) = (L_r(\mathbb{R},X_0),L_r(\mathbb{R},X_1))_{\theta,r}$
we get that for an interpolation couple of spaces of class $\HT$ $X_0,X_1$ the real interpolation space $(X_0,X_1)_{\theta,r}$ is also of class $\HT$. Thus $F^{s,r}_{p,q}$ is of class $\HT$.
\end{proof}

\begin{korollar} \label{thm: TLL reflexive}
$F^{s,r}_{p,q}$ is reflexive for $s \in \mathbb{R}$ and $1 < p,q,r < \infty$ (due to~\cite{RubioDeFrancia}).
\end{korollar}

Corollary~\ref{thm: TLL reflexive} could also be obtained in a direct way, regarding the following result which is a conclusion of the corresponding result for Triebel-Lizorkin spaces (see~\cite[Thm. 2.6.2]{Triebel1978}) and Theorem~\ref{thm: Interpolation von TLL-Räumen}.

\begin{satz} \label{thm: Dualraum TLL}
The dual space to $F^{s,r}_{p,q}$ is given by $F^{-s,r'}_{p',q'}$ for $s \in \mathbb{R}$ and $1 < p,q,r < \infty$ where $1 < p',q',r' < \infty$ are given by $\frac{1}{p} + \frac{1}{p'} = 1$, $\frac{1}{q} + \frac{1}{q'} = 1$ and $\frac{1}{r} + \frac{1}{r'} = 1$.
\end{satz}

\begin{satz} \label{thm: F^(s,r)_(p,q) property alpha}
$F^{s,r}_{p,q}$ has $\Propa$ for $s \in \mathbb{R}$ and $1 < p,q,r < \infty$.
\end{satz}

\begin{proof}
The Triebel-Lizorkin spaces $F^s_{p,q}$ have $\Propa$ since there exists a continuous embedding in $L_p(\mathbb{R}^n,l^s_q)$. This implies the assertion, since $\Propa$ preserves under real interpolation. We refer to~\cite[Thm.\ 4.5]{Kaip}.

\end{proof}

\begin{theorem}[Multiplier theorem for Triebel-Lizorkin-Lorentz spaces] \label{thm: Multiplier theorem}
Let $s \in \mathbb{R}$, $1 < p,q < \infty$ and $1 \le r \le \infty$.
Let $(m_\lambda)_{\lambda \in \Lambda} \subset C^n(\mathbb{R}^n \setminus \{ 0 \}, \mathbb{C})$ such that
$C_\alpha := \sup_{\xi \in \mathbb{R}^n \setminus \{ 0 \}, \lambda \in \Lambda} |\xi^\alpha \partial^\alpha m_\lambda(\xi)| < \infty$ for all $\alpha \in \{ 0,1 \}^n$. Then for every $\lambda \in \Lambda$
\begin{equation*}
\Fourier^{-1} m_\lambda \Fourier: \Schwartz \longrightarrow \Temp
\end{equation*}
has a (unique) continuous extension
$T_\lambda: F^{s,r}_{p,q} \longrightarrow F^{s,r}_{p,q}$ such that
\begin{equation*}
\| T_\lambda \|_{\mathscr{L}(F^{s,r}_{p,q})}
\le C \max_{\alpha \in \{ 0,1 \}^n} C_\alpha,
\end{equation*}
where the constant $C > 0$ only depends on $n$ and the parameters $p,q,s,r$.
Furthermore, $(T_\lambda)_{\lambda \in \Lambda} \subset \mathscr{L}(F^{s,r}_{p,q})$ is $\RR$-bounded in the case $1 < r < \infty$.
\end{theorem}

\begin{proof}
We define $M_\lambda \in L_\infty(\mathbb{R}^n,\mathscr{L}(l^s_q))$ by setting $M_\lambda(\xi)x := (m_\lambda(\xi)x_k)_{k \in \mathbb{N}_0}$ for $\xi \in \mathbb{R}^n \setminus \{ 0 \}$, $x = (x_k)_{k \in \mathbb{N}_0} \in l^s_q$ and $\lambda \in \Lambda$. By Kahane's contraction principle we see that the assumption $C_\alpha < \infty$ implies the $\RR$-boundedness of $\{ \xi^\alpha \partial^\alpha M_\lambda(\xi) : \xi \in \mathbb{R}^n \setminus \{ 0 \}, ~\lambda \in \Lambda \} \subset \mathscr{L}(l^s_q)$ and the $\RR_q$-bound doesn't exceed $2 \max_{\alpha \in \{ 0,1 \}^n} C_\alpha$. Since $l^s_q$ is of class $\HT$ (note that $1 < q < \infty$) with $\Propa$, Theorem~\ref{thm: Satz von Mikhlin operatorwertig} gives that $M_\lambda$ is a Fourier multiplier, i.e.,
\begin{equation*}
\Fourier^{-1} M_\lambda \Fourier: \mathscr{S}(\mathbb{R}^n,l^s_q) \longrightarrow \mathscr{S}'(\mathbb{R}^n,l^s_q)
\end{equation*}
has a (unique) continuous extension
$S_\lambda: L_p(l^s_q) \longrightarrow L_p(l^s_q)$ such that
\begin{equation} \label{eq: R-Schranke}
\RR_q(\{ S_\lambda : \lambda \in \Lambda \})  \le C \max_{\alpha \in \{ 0,1 \}^n} C_\alpha =: K
\end{equation}
for all $\lambda \in \Lambda$.
From the identity
\begin{equation} \label{eq: Zusammenhang von m_lambda und M_lambda}
(\varphi_k * \Fourier^{-1} m_\lambda \Fourier f)_{k \in \mathbb{N}_0}
= \Fourier^{-1} M_\lambda \Fourier (\varphi_k * f)_{k \in \mathbb{N}_0}
\end{equation}
we get $\| \Fourier^{-1} m_\lambda \Fourier f \|_{F^s_{p,q}} \le K \| f \|_{F^s_{p,q}}$ for $f \in \Schwartz$ and consequently we have a
continuous extension $T_\lambda: F^s_{p,q} \longrightarrow F^s_{p,q}$ ~of~ $\Fourier^{-1} m_\lambda \Fourier: \Schwartz \longrightarrow \Temp$.
Now~\eqref{eq: R-Schranke} and~\eqref{eq: Zusammenhang von m_lambda und M_lambda} imply $\RR_q(\{ T_\lambda : \lambda \in \Lambda \}) \le K$. Hence the assertion is proved in the case $p = r$.

In order to generalize the result, we select $1 < p_0 < p < p_1 < \infty$ and $0 < \theta < 1$ such that $\frac{1}{p} = \frac{1 - \theta}{p_0} + \frac{\theta}{p_1}$ and get $F^{s,r}_{p,q} = (F^s_{p_0,q},F^s_{p_1,q})_{\theta,r}$.
Thus for
\begin{equation*}
T_\lambda: (F^s_{p_0,q},F^s_{p_1,q})_{\theta,r} \longrightarrow (F^s_{p_0,q},F^s_{p_1,q})_{\theta,r}
\end{equation*}
we get the estimate $\| T_\lambda \|_{\mathscr{L}(F^{s,r}_{p,q})} \le C' \max_{\alpha \in \{ 0,1 \}^n} C_\alpha$ since the real interpolation method is exact of type $\theta$, where $C'>0$ is also a constant depending only on $s,p,q,r$ and $n$.

Since $F^s_{p_j,q}$ ist of class $\HT$ for $j=0,1$ we get the $\RR$-boundedness of
\begin{equation*}
T_\lambda: (F^s_{p_0,q},F^s_{p_1,q})_{\theta,r} \longrightarrow (F^s_{p_0,q},F^s_{p_1,q})_{\theta,r}
\end{equation*}
for $1 < r < \infty$ as a consequence of the case $p = r$ proved above (see \cite[Thm. 3.19]{Kaip}). 
\end{proof}

\begin{satz} \label{thm: Darstellung Ableitungen}
For $s \in \mathbb{R}$, $1 < p,q < \infty$ and $1 \le r \le \infty$ the following representations hold.
\begin{enumerate}[(i)]
\item \label{thm: Darstellung Ableitungen 1} $F^{s+\sigma,r}_{p,q} = \big\{ u \in \Temp : \Fourier^{-1} (1 + |\xi|^2)^\frac{\sigma}{2} \Fourier u \in F^{s,r}_{p,q} \}$ for $\sigma \in \mathbb{R}$.
\item \label{thm: Darstellung Ableitungen 2} $F^{s+k,r}_{p,q} = \big\{ u \in \Temp : \partial^\alpha u \in F^{s,r}_{p,q} ~\forall \alpha \in \mathbb{N}_0^n, |\alpha| \le k \big\}$ for $k \in \mathbb{N}_0$.
\item \label{thm: Darstellung Ableitungen 3} $F^{s+2m,r}_{p,q} = \big\{ u \in \Temp : \Delta^j u \in F^{s,r}_{p,q} ~\forall j \in \mathbb{N}_0, j \le m \big\}$ for $m \in \mathbb{N}_0$.
\end{enumerate}
The corresponding norms are eqivalent, where the norm of the space on the right-hand side is given by $\| \Fourier^{-1} (1 + |\xi|^2)^\frac{\sigma}{2} \Fourier u \|_{F^{s,r}_{p,q}}$ in~\eqref{thm: Darstellung Ableitungen 1}, by $\sum_{|\alpha| \le k} \| \partial^\alpha u \|_{F^{s,r}_{p,q}}$ in~\eqref{thm: Darstellung Ableitungen 2} and by $\sum_{0 \le j \le m} \| \Delta^j u \|_{F^{s,r}_{p,q}}$ in~\eqref{thm: Darstellung Ableitungen 3}.
\end{satz}

\begin{proof}
We consider the Bessel-potential operator $B^\sigma u := \Fourier^{-1} (1 + |\xi|^2)^\frac{\sigma}{2} \Fourier u$ for $u \in \Temp$ and $\sigma \in \mathbb{R}$.
If we fix $(\varphi_k)_{k \in \mathbb{N}_0} \in \boldsymbol{\Phi}$ and $\sigma \in \mathbb{R}$, then by setting $\widehat{\psi}_k(\xi) = \frac{2^{k \sigma}}{(1 + |\xi|^2)^\frac{\sigma}{2}} \widehat{\varphi}_k(\xi)$ we get $(\psi_k)_{k \in \mathbb{N}_0} \in \boldsymbol{\Phi}$ (see also the proof of \cite[Thm. 2.3.4]{Triebel1978}). Hence
\begin{equation} \label{eq: Bessel-potential}
\| u \|_{F^{s,r}_{p,q}} \sim \| B^\sigma u \|_{F^{s - \sigma,r}_{p,q}}
\end{equation}
and we get~\eqref{thm: Darstellung Ableitungen 1}.

Now the special case $\sigma = 2m$ in~\eqref{eq: Bessel-potential} leads to $F^{s+2m,r}_{p,q} = \big\{ u \in \Temp : (1 - \Delta)^m u \in F^{s,r}_{p,q} \big\}$ together with the equivalence $\| u \|_{F^{s+2m,r}_{p,q}} \sim \| (1 - \Delta)^m u \|_{F^{s,r}_{p,q}}$.
Hence for~\eqref{thm: Darstellung Ableitungen 3} it remains to show
$\sum_{0 \le j \le m} \| \Delta^j u \|_{F^{s,r}_{p,q}} \le C \| (1 - \Delta)^m u \|_{F^{s,r}_{p,q}}$ since the converse estimate is obvious.
For this purpose we write
\begin{equation*}
(- \Delta)^j u = \Fourier^{-1} \frac{|\xi|^{2j}}{(1 + |\xi|^2)^k} \Fourier (1 - \Delta)^k u.
\end{equation*}
Now the associated symbol $\frac{|\xi|^{2j}}{(1 + |\xi|^2)^k}$ fulfills the conditions of Theorem~\ref{thm: Multiplier theorem} and we get the assertion.

In order to verify~\eqref{thm: Darstellung Ableitungen 2} we write
\begin{equation} \label{eq: Symbolschreibweise 1}
\partial^\alpha u = i^{|\alpha|} \Fourier^{-1} \frac{\xi^\alpha}{(1 + |\xi|)^\frac{|\alpha|}{2}} \Fourier B^{|\alpha|} u \quad \text{for } |\alpha| \le k
\end{equation}
and
\begin{equation} \label{eq: Symbolschreibweise 2}
B^k u = \Fourier^{-1} \sum_{|\alpha| \le k} \frac{k!}{\alpha! (k - |\alpha|)!} \xi^\alpha \frac{\xi^\alpha}{(1 + |\xi|^2)^\frac{k}{2}} \Fourier u.
\end{equation}
Now using again Theorem~\ref{thm: Multiplier theorem} and~\eqref{eq: Bessel-potential} we get $\| u \|_{F^{s+k,r}_{p,q}} \sim \sum_{|\alpha| \le k} \| \partial^\alpha u \|_{F^{s,r}_{p,q}}$ where~\eqref{eq: Symbolschreibweise 2} gives the estimate ''$\le$'' and~\eqref{eq: Symbolschreibweise 1} gives ''$\ge$''. 
\end{proof}

\section{The Laplace operator in $F^{s,r}_{p,q}$} \label{secLap}

The Laplace operator in $F^{s,r}_{p,q}$ for $s \in \mathbb{R}$, $1 < p,q < \infty$ and $1 \le r \le \infty$ is defined as
\begin{equation*}
A_L = A^{s,r}_{L,p,q}: \Def(A_L) \subset F^{s,r}_{p,q} \longrightarrow F^{s,r}_{p,q},
\quad u \longmapsto - \Delta u,
\end{equation*}
where the domain is $\Def(A_L) = F^{s+2,r}_{p,q}$.

\begin{satz} \label{thm: R-Sektorialitaet Laplace-Operator}
$A_L$ is $\RR$-sectorial with $\varphi^\RR_{A_L} = 0$ for $s \in \mathbb{R}$ and $1 < p,q,r < \infty$ and we have
$(A^{s,r}_{L,p,q})' = A^{-s,r'}_{L,p',q'}$.
\end{satz}


\begin{proof}
Lemma \ref{thm: Einbettungen}~\eqref{thm: Einbettungen 1} implies that $A_L$ is densely defined.
For $\lambda \in \mathbb{C} \setminus (-\infty,0]$ we would like to have $\lambda \in \rho(-A_L)$ with
\begin{equation} \label{eq: Resolvente}
(\lambda + A_L)^{-1} = \Fourier^{-1} \frac{1}{\lambda + |\xi|^2} \Fourier.
\end{equation}
Therefore, we consider the symbols $\frac{1}{\lambda + |\xi|^2}$ and $\frac{|\xi|^2}{\lambda + |\xi|^2}$, which are smooth and fulfill the conditions of Theorem~\ref{thm: Multiplier theorem}. From the first symbol we get that~\eqref{eq: Resolvente} defines a bounded operator on $F^{s,r}_{p,q}$. From the second symbol and Proposition~\ref{thm: Darstellung Ableitungen} we get that~\eqref{eq: Resolvente} has in fact values in $F^{s+2,r}_{p,q}$ and hence must be the inverse operator of $\lambda + A_L$.

To prove the claimed $\RR$-boundedness of $\{ \lambda(\lambda + A_L)^{-1} : \lambda \in \Sigma_\varphi \} \subset \mathscr{L}(F^{s,r}_{p,q})$, we need the uniform estimate
\begin{equation*}
\sup_{\xi \in \mathbb{R}^n, ~\lambda \in \Sigma_\varphi} |\xi^\alpha \partial^\alpha m_\lambda(\xi)| < \infty
\end{equation*}
for all $\alpha \in \mathbb{N}_0^n$ and $\varphi < \pi$, where $m_\lambda(\xi) := \frac{\lambda}{\lambda + |\xi|^2}$. This is a consequence of Lemma~\ref{thm: Symbolabschaetzung}~\eqref{thm: Symbolabschaetzung 1}, so we can apply Theorem~\ref{thm: Multiplier theorem}.
Summarizing, $A_L$ is pseudo-$\RR$-sectorial with $\varphi_{A_L}^\RR = 0$.

Let now initially $s > -2$. Then we obtain in an elementary way that $A_L$ is injective: For $u \in \mathcal{N}(A_L)$ we have $\Supp(\widehat{u}) \subset \{ 0 \}$ and thus $u$ is a polynomial (see e.g. \cite[Cor. 2.4.2]{Gra1_3rd}). Lemma \ref{thm: Einbettungen} gives that $F^{s+2,r}_{p,q} \subset L_{p,\infty}$ and it is not hard to show that $L_{p,\infty}$ doesn't contain any nontrivial polynomials. Hence $u = 0$.
Now we consider the decomposition $F^{s,r}_{p,q} = \mathscr{N}(A_L) \oplus \overline{\mathscr{R}(A_L)}$, which is a consequence of the pseudo-$\RR$-sectoriality proved above and of the reflexivity of $F^{s,r}_{p,q}$ obtained in Corollary~\ref{thm: TLL reflexive} (see e.g. \cite[Prop. 2.1.1]{Haase}).
The injectivity of $A_L$ then gives the density of $\mathscr{R}(A_L) \subset F^{s,r}_{p,q}$.

By integration by parts we easily obtain
$A^{-s,r'}_{L,p',q'}\subset (A^{s,r}_{L,p,q})'$. The fact that 
$1\in\rho(A^{s,r}_{L,p,q})$ for all $s\in\mathbb{R}$ and $1<p,q,r<\infty$ then
gives $A^{-s,r'}_{L,p',q'}= (A^{s,r}_{L,p,q})'$. Since 
$(F^{s,r}_{p,q})'=F^{-s,r'}_{p',q'}$, the $\RR$-sectoriality with 
$\varphi^\RR_{A_L}=0$ for $s\le -2$ now follows by standard permanence properties
for $\RR$-sectorial operators.
\end{proof}

\begin{bemerkung} \label{thm: Injektivität Laplace}
The proof of Proposition~\ref{thm: R-Sektorialitaet Laplace-Operator} shows that for $r \in \{ 1,\infty \}$ we still have that $A_L$ is pseudo-sectorial with $\varphi_{A_L} = 0$ and, in the case $s > -2$, $A_L$ is injective.
\end{bemerkung}

\begin{satz} \label{thm: H^infty Kalkül Laplace}
Let $s \in \mathbb{R}$ and $1 < p,q,r < \infty$.
Then $A_L$ has an $\RR$-bounded $H^\infty$-calculus with $\varphi_{A_L}^{\RR,\infty}$ = 0.
\end{satz}

\begin{proof}
Thanks to Proposition~\ref{thm: F^(s,r)_(p,q) property alpha} and Theorem~\ref{thm: H^infty und Propa impliziert RH^infty} it is sufficient to prove that $A_L$ has a bounded $H^\infty$-calculus with $\varphi_{A_L}^\infty$ = 0.
Let $\varphi \in (0,\pi)$ and $f \in \mathscr{H}_0(\Sigma_\varphi)$.
$A_L$ is sectorial thanks to Proposition~\ref{thm: R-Sektorialitaet Laplace-Operator}.
Using Cauchy's integral formula we get
$f(A_L)u = \Fourier^{-1} f(|\xi|^2) \Fourier u$
for all $u \in \Schwartz$. Now the symbol $f(|\xi|^2)$ fulfills the condition of Theorem~\ref{thm: Multiplier theorem} (due to Lemma~\ref{thm: Symbolabschaetzung}~\eqref{thm: Symbolabschaetzung 1})
so we have
\begin{equation*}
\| f(A_L) \|_{\mathscr{L}(F^{s,r}_{p,q})}
\le C_\varphi \| f \|_{L_\infty(\Sigma_\varphi)}.
\end{equation*}
\end{proof}

Note that Proposition~\ref{thm: H^infty Kalkül Laplace}
implies~\ref{thm: R-Sektorialitaet Laplace-Operator} if we only knew the
sectoriality of $A_L$. But, as the proof of
Proposition~\ref{thm: R-Sektorialitaet Laplace-Operator} shows, $\RR$-sectoriality 
can be obtained in a direct way at essentially the same cost.

Now we consider an alternative representation for Triebel-Lizorkin-Lorentz spaces. We would like to prove that $F^{s + 2 \alpha,r}_{p,q}$ is the domain of $(1 - \Delta)^\alpha$ in $F^{s,r}_{p,q}$, where $\alpha \in [0,1]$.

\begin{satz} \label{thm: alpha Potenz von A}
Let $s \in \mathbb{R}$, $1 < p,q < \infty$ and $1 \le r \le \infty$. Then
\begin{equation*}
\A: \Def(\A) = \Def(A_L) \subset F^{s,r}_{p,q} \longrightarrow F^{s,r}_{p,q}, \quad
u \longmapsto (1 - \Delta)u
\end{equation*}
is sectorial with angle $\varphi_\A = 0$ and for $\alpha \in [0,1]$
\begin{equation} \label{eq: Definitionsbereich A^alpha}
\Def(\A^\alpha) = \big\{ u \in F^{s,r}_{p,q} ~|~ \Fourier^{-1} (1 + |\xi|^2)^\alpha \Fourier u \in F^{s,r}_{p,q} \big\}
= F^{s + 2 \alpha,r}_{p,q}
\end{equation}
holds with equivalent norms, i.e., $\| u \|_{\Def(\A^\alpha)} \sim \| \Fourier^{-1} (1 + |\xi|^2)^\alpha \Fourier u \|_{F^{s,r}_{p,q}}$ for all $u \in \Def(\A^\alpha)$. Moreover, we have
\begin{equation} \label{eq: Darstellung A^alpha}
\A^\alpha u = \Fourier^{-1} (1 + |\xi|^2)^\alpha \Fourier u
\end{equation}
for all $u \in \Def(\A^\alpha)$.
\end{satz}



\begin{proof}
The second equality in~\eqref{eq: Definitionsbereich A^alpha} is Proposition~\ref{thm: Darstellung Ableitungen}~\eqref{thm: Darstellung Ableitungen 1}.
The Laplace operator $A_L$ is pseudo-sectorial with angle $\varphi_{A_L} = 0$ and so is $\A$. Now $-1 \in \rho(A_L)$, so $\A$ is bijective and thus sectorial.

We now assume $\alpha \in (0,1)$ since the cases $\alpha = 0$ and $\alpha = 1$ are obvious.
We set $g(z):= \frac{z}{(1 + z)^2}$ and $h_\alpha(z) := z^\alpha$. Using Cauchy's integral formula, we obtain
\begin{equation} \label{eq: (g h_alpha)(A)f}
(g h_\alpha)(\A)f = \Fourier^{-1} \frac{(1 + |\xi|^2)^{\alpha + 1}}{(2 + |\xi|^2)^2} \Fourier f
\end{equation}
for all $f \in \Schwartz$. Theorem~\ref{thm: Multiplier theorem} gives that~\eqref{eq: (g h_alpha)(A)f} even holds for all $f \in F^{s,r}_{p,q}$. Now $\A: \Def(\A) \rightarrow F^{s,r}_{p,q}$ is bijective with
$\A^{-1} f = \Fourier^{-1} \frac{1}{1 + |\xi|^2} \Fourier f$ for $f \in F^{s,r}_{p,q}$ and thus we get
\begin{equation} \label{eq: g(A)^-1}
g(\A)^{-1} f
= (1 + \A)^2 \A^{-1} f
= \Fourier^{-1} \frac{(2 + |\xi|^2)^2}{1 + |\xi|^2} \Fourier f
\end{equation}
for all $f \in \Def(\A) \cap \Bild(\A)$. Relations~\eqref{eq: (g h_alpha)(A)f} and~\eqref{eq: g(A)^-1} yield~\eqref{eq: Darstellung A^alpha}
since $\A^\alpha$ is given by $g(\A)^{-1} (g h_\alpha)(\A)$.

Now we verify~\eqref{eq: Definitionsbereich A^alpha} together with the equivalence of the norms. For this purpose let first $u \in F^{s,r}_{p,q}$ so that $\Fourier^{-1} (1 + |\xi|^2)^\alpha \Fourier u \in F^{s,r}_{p,q}$. Then, using~\eqref{eq: (g h_alpha)(A)f} and Theorem~\ref{thm: Multiplier theorem}, we get
$(2 - \Delta)(g h_\alpha)(\A)u \in F^{s,r}_{p,q}$.
Consequently, we have $(g h_\alpha)(\A)u \in \Def(\A)$.
Now, again using~\eqref{eq: (g h_alpha)(A)f}, we can also write $(g
h_\alpha)(\A)u = (1 - \Delta)v$, where $v := \Fourier^{-1} \frac{(1 +
|\xi|^2)^\alpha}{(2 + |\xi|^2)^2} \Fourier u$. Then Theorem~\ref{thm:
Multiplier theorem} gives $v \in \Def(\A)$ and thus $(g h_\alpha)(\A)u
\in \Bild(\A)$. Summarizing, we obtain $u \in \Def(\A^\alpha)$. Hence we have $\big\{ u \in F^{s,r}_{p,q} ~|~ \Fourier^{-1} (1 + |\xi|^2)^\alpha \Fourier u \in F^{s,r}_{p,q} \big\} \subset \Def(\A^\alpha)$, so we can restrict ourselves to $u \in \Def(\A^\alpha)$ to show the equivalence
\begin{equation} \label{eq: Äquivalenz der Normen}
\| u \|_{\Def(\A^\alpha)} 
= \| u \|_{F^{s,r}_{p,q}} + \| \A^\alpha u \|_{F^{s,r}_{p,q}}
\sim \| \Fourier^{-1} (1 + |\xi|^2)^\alpha \Fourier u \|_{F^{s,r}_{p,q}}.
\end{equation}
For $u \in \Def(\A^\alpha)$ we can apply~\eqref{eq: Darstellung
A^alpha}, so we directly get "$\ge$" in~\eqref{eq: Äquivalenz der
Normen}. Applying Theorem~\ref{thm: Multiplier theorem} to the symbol $\frac{1}{(1 + |\xi|^2)^\alpha}$ and using~\eqref{eq: Darstellung A^alpha} again, we get the estimate
$\| u \|_{F^{s,r}_{p,q}} \le C \| \Fourier^{-1} (1 + |\xi|^2)^\alpha \Fourier u \|_{F^{s,r}_{p,q}}$ and consequently the converse inequality in~\eqref{eq: Äquivalenz der Normen}. Hence we have proved the equivalence~\eqref{eq: Äquivalenz der Normen} and this also shows 
$\Def(\A^\alpha) \subset \big\{ u \in F^{s,r}_{p,q} ~|~ \Fourier^{-1} (1 + |\xi|^2)^\alpha \Fourier u \in F^{s,r}_{p,q} \big\}$.
\end{proof}

As a consequence of Propositions~\ref{thm: H^infty Kalkül Laplace} and~\ref{thm: alpha Potenz von A} and of~\eqref{eq: H^infty und komplexe Interpolation} we also get the following result on complex interpolation of Triebel-Lizorkin-Lorentz spaces.

\begin{korollar}
Let $-\infty < s_0 \le s_1 < \infty$ and $1 < p,q,r < \infty$. Then for $\eta \in (0,1)$ we have
\begin{equation*}
[F^{s_0,r}_{p,q},F^{s_1,r}_{p,q}]_{\eta} = F^{(1-\eta)s_0 + \eta s_1,r}_{p,q}.
\end{equation*}
\end{korollar}

\begin{proof}
For $s \in \mathbb{R}$ we get
\begin{equation} \label{eq: komplexe Interpolation TLL}
[F^{s,r}_{p,q},F^{s+2k \theta,r}_{p,q}]_{\eta} = F^{s+2k \theta \eta,r}_{p,q}
\end{equation} in the case $k=1$, $\theta = 1$ from Propositions~\ref{thm: H^infty Kalkül Laplace} and~\ref{thm: alpha Potenz von A}. Since for any $\beta \ge 0$ we can write $\A^\beta = \A^m \A^\alpha$ for some $m \in \mathbb{N}_0$ and $\alpha \in [0,1]$, \eqref{eq: komplexe Interpolation TLL} holds for all $k \in \mathbb{N}_0$ and $\theta = 1$. Application of the reiteration theorem now gives~\eqref{eq: komplexe Interpolation TLL} for all $\theta \in [0,1]$ and $k \in \mathbb{N}_0$. This proves the claim.
\end{proof}


\section{The Stokes operator in $F^{s,r}_{p,q}$}
\label{sec: Helmholtz Projection and Stokes Operator}

We first introduce the Helmholtz projection on $(F^{s,r}_{p,q})^n$. Again $n \in \mathbb{N}$ is the dimension and $s \in \mathbb{R}$, $1 < p,q < \infty$, $1 \le r \le \infty$. For $u \in \Schwartz^n$ we set
\begin{equation*}
Pu := \Fourier^{-1} \bigg[ 1 - \frac{\xi \xi^T}{|\xi|^2} \bigg] \Fourier u
= u - \bigg( \sum_{j=1}^n \Fourier^{-1} \frac{\xi_i \xi_j}{|\xi|^2} \Fourier u_j \bigg)_{1 \le i \le n}.
\end{equation*}
From Theorem~\ref{thm: Multiplier theorem} we get the \emph{Helmholtz projection} as the bounded extension $P \in \mathscr{L}((F^{s,r}_{p,q})^n)$. The space of solenoidal functions is
\begin{equation*}
(F^{s,r}_{p,q})^n_\sigma := \big\{ u \in (F^{s,r}_{p,q})^n ~|~ \mathrm{div}~ u = 0 \big\}
\end{equation*}
and the space of gradient fields in $(F^{s,r}_{p,q})^n$ is
\begin{equation*}
\grad := \big\{ \nabla p ~|~ p \in \mathscr{D}'(\mathbb{R}^n), \nabla p \in (F^{s,r}_{p,q})^n \big\}.
\end{equation*}
Furthermore, let $C_c^\infty(\mathbb{R}^n)^n_\sigma$ denote the smooth functions with compact support and vanishing divergence. Now we get the Helmholtz decomposition:

\begin{satz} \label{thm: hhd}
Similar to the Definition of the space of gradient fields we set
\begin{equation*}
\gradd := \big\{ \nabla p ~|~ p \in \Temp, \nabla p \in (F^{s,r}_{p,q})^n \big\}.
\end{equation*}
Let $1 < p,q,r < \infty$ and $n \ge 2$.
If $s > -2$, we additionally admit $r \in \{ 1,\infty \}$.
Then range and nullspace of the Helmholtz projection are given by $\Bild(P) = (F^{s,r}_{p,q})^n_\sigma$ and $\Kern(P) = \grad = \widebar{\gradd}$. In particular the Helmholtz decomposition
\begin{equation*}
(F^{s,r}_{p,q})^n = (F^{s,r}_{p,q})^n_\sigma ~\oplus~ \grad
\end{equation*}
holds.
\end{satz}

\begin{proof}
We prove the claim in three steps and start with some general observations that we
will make use of. First we remark that one gets the inclusion $\Bild(P) \subset (F^{s,r}_{p,q})^n_\sigma$ by direct computation (and approximation). Second the injectivity of the Laplace operator (see Proposition~\ref{thm: R-Sektorialitaet Laplace-Operator} and Remark~\ref{thm: Injektivität Laplace}) yields
\begin{equation} \label{eq: leerer Schnitt}
(F^{s,r}_{p,q})^n_\sigma \cap \grad = \{ 0 \}.
\end{equation}
Furthermore, de Rham's theorem (see~\cite{Galdi} and the references therein)
gives that $\grad$ is a closed subspace of $(F^{s,r}_{p,q})^n$.

\emph{Step 1.}
We show $\Kern(P) \subset \gradd$ in the special case that
$F^{s,r}_{p,q}$ is a Lebesgue space. So, we fix $1 < \eta < 2$ and set $\gradd_\eta := \big\{ \nabla p ~|~ p \in \Temp, \nabla p \in L_\eta(\mathbb{R}^n)^n \big\}$. Furthermore, let $P_\eta$ denote the Helmholtz projection on $L_\eta(\mathbb{R}^n)^n$. Then the Hausdorff-Young theorem gives that $P_\eta u = \Fourier^{-1} \big[ 1 - \frac{\xi \xi^T}{|\xi|^2} \big] \Fourier u$ (which is apriori valid for Schwartz functions) is meaningful for all $u \in L_\eta(\mathbb{R}^n)^n$. Thus for $u \in \Kern(P_\eta)$ we have $\widehat{u} = \xi \frac{\xi^T}{|\xi|^2} \widehat{u}$ with
$\frac{\xi^T}{|\xi|^2} \widehat{u} \in \Temp$
(since $n \ge 2$) and get $\Kern(P_\eta) \subset \gradd_\eta$.

\emph{Step 2.} We use the first step to show the inclusions $\Kern(P) \subset \widebar{\gradd} \subset \grad \subset \Kern(P)$ (in the stated order). For a fixed $1 < \eta < 2$ we get $(1 - P)((F^{s,r}_{p,q})^n \cap L_\eta(\mathbb{R}^n)^n) \subset \gradd$ from the first step. Since $(F^{s,r}_{p,q})^n \cap L_\eta(\mathbb{R}^n)^n$ is dense in $(F^{s,r}_{p,q})^n$ we get
$\Kern(P) = (1 - P)((F^{s,r}_{p,q})^n) \subset \widebar{\gradd}$ as a conclusion. The second inclusion $\widebar{\gradd} \subset \grad$ is valid since $\grad$ is closed. For the third inclusion we fix $u \in \grad$. Since we have already shown $\Kern(P) \subset \widebar{\gradd} \subset \grad$, we obtain $Pu = u - (1 - P)u \in \grad$.
On the other hand, we have $Pu \in \Bild(P) \subset (F^{s,r}_{p,q})^n_\sigma$. Consequently,~\eqref{eq: leerer Schnitt} implies $Pu = 0$.

\emph{Step 3.}
It remains to prove $\Bild(P) = (F^{s,r}_{p,q})^n_\sigma$. In view of what we have already seen we get $(F^{s,r}_{p,q})^n = \Bild(P) \oplus \Kern(P) = \Bild(P) \oplus \grad \subset (F^{s,r}_{p,q})^n_\sigma + \grad$. Now the last inclusion is an equality since the converse inclusion is obvious. Besides, \eqref{eq: leerer Schnitt} yields the directness of the sum, so $\Bild(P) \oplus \grad = (F^{s,r}_{p,q})^n_\sigma \oplus \grad$ together with $\Bild(P) \subset (F^{s,r}_{p,q})^n_\sigma$ gives $\Bild(P) = (F^{s,r}_{p,q})^n_\sigma$.
\end{proof}

\begin{bemerkung} \label{thm: UMD divergenzfreier F^s,r_p,q}
The space $(F^{s,r}_{p,q})^n_\sigma$ is of class $\HT$ for $1 < p,q,r < \infty$ and $s \in \mathbb{R}$. This is a consequence of Proposition~\ref{thm: F^(s,r)_(p,q) UMD-Raum}: $F^{s,r}_{p,q}$ is of class $\HT$, so is $(F^{s,r}_{p,q})^n$ and hence $(F^{s,r}_{p,q})^n_\sigma$ as a closed subspace.
\end{bemerkung}

Now we are able to define the \emph{Stokes operator} as
\begin{equation*}
A_S = A^{s,r}_{S,p,q}: \Def(A_S) \subset (F^{s,r}_{p,q})^n_\sigma \longrightarrow (F^{s,r}_{p,q})^n_\sigma, \quad
u \longmapsto -P \Delta u
\end{equation*}
on the domain $\Def(A_S) := (F^{s+2,r}_{p,q})^n_\sigma$.

\begin{satz} \label{thm: Stokes Einschraenkung von Laplace}
For $s \in \mathbb{R}$ and $1 < p,q,r < \infty$ we have
$A_S = A_L|_{\Def(A_S)}$. Besides, we have $\rho(A_L) \subset \rho(A_S)$ with $(\lambda - A_S)^{-1} = (\lambda - A_L)^{-1}|_{(F^{s,r}_{p,q})^n_\sigma}$ for all $\lambda \in \rho(A_L)$.
\end{satz}

\begin{proof}
For $u \in \Def(A_S)$ we get $u \in \mathscr{N}(1 - P)$ since $P$ is a projection and thus $Pu = u$.
By Proposition~\ref{thm: Darstellung Ableitungen} and the continuity of $\Delta: \Temp \rightarrow \Temp$ we get $P \Delta = \Delta P$ on $(F^{s+2,r}_{p,q})^n$. This shows $A_S u = A_L u$ for $u \in \Def(A_S)$.

Now let $\lambda \in \rho(A_L)$ and set $T_\lambda v := (\lambda -
A_L)^{-1}v$ for $v \in (F^{s,r}_{p,q})^n_\sigma$. Again we can use $P
\Delta = \Delta P$ and get $P T_\lambda = T_\lambda$ by the injectivity
of $\lambda - A_L$, i.e $T_\lambda$ maps into
$(F^{s,r}_{p,q})^n_\sigma$. Consequently, $T_\lambda = (\lambda - A_S)^{-1}$ on  $(F^{s,r}_{p,q})^n_\sigma$.
\end{proof}

\begin{satz} \label{thm: Stokesoperator MR}
Let $s \in \mathbb{R}$ and $1 < p,q,r < \infty$. Then $A_S$ is $\RR$-sectorial with $\varphi_{A_S} = 0$. Hence $A_S \in \mathrm{MR}((F^{s,r}_{p,q})^n_\sigma)$.
Furthermore, we have 
$(A^{s,r}_{S,p,q})'=A^{-s,r'}_{S,p',q'}$.
\end{satz}

\begin{proof}
Let $0 < \varphi < \pi$. Then we get the $\RR$-boundedness of $\{
	\lambda(\lambda + A_S)^{-1} : \lambda \in \Sigma_\varphi \}
	\subset \mathscr{L}((F^{s,r}_{p,q})^n_\sigma)$ as a direct consequence of Propositions~\ref{thm: R-Sektorialitaet Laplace-Operator} and~\ref{thm: Stokes Einschraenkung von Laplace}.  $(F^{s,r}_{p,q})^n_\sigma$ is reflexive (see Corollary~\ref{thm: TLL reflexive}). Consequently, we get the density of $\Def(A_S) \subset (F^{s,r}_{p,q})^n_\sigma$ (see e.g.~\cite[Prop. 2.1.1]{Haase}).
The Laplace operator $A_L$ is injective and so is $A_S$.
The remaining proof is thus completely analogous to the proof of
Proposition~\ref{thm: R-Sektorialitaet Laplace-Operator}.
\end{proof}

\section{The time derivative $\timeD$ and some embeddings} \label{sec: time der}

We take a look at the time derivative operator $\timeD$, more precisely at
\begin{equation*}
B: \Def(B) = H^1_p(\mathbb{R},X) \subset L_p(\mathbb{R},X) \longrightarrow L_p(\mathbb{R},X), \quad u \longmapsto (1 + \timeD) u.
\end{equation*}

\begin{satz} \label{thm: Definitionsbereich Potenz der Zeitableitung}
Let $1 < p < \infty$ and let $X$ be of class $\HT$ with $\Propa$. Then $B$ is sectorial with angle $\varphi_B = \frac{\pi}{2}$ and we have
\begin{equation}
\Def(B^\alpha) = H^\alpha_p(\mathbb{R},X)
\end{equation}
for $\alpha \in [0,1]$. The related norms are equivalent.
Furthermore, we have
\begin{equation}
B^\alpha u = \Fourier^{-1} (1 + i \xi)^\alpha \Fourier u \quad \forall u \in \Def(B^\alpha).
\end{equation}
\end{satz}

\begin{satz} \label{thm: H^infty Kalkül für Zeitableitung}
Let $X$ be a Banach space of class $\HT$ with $\Propa$ and $1 < p < \infty$. Then $B$
has an $\RR$-bounded $H^\infty$-calculus in $L_p(\mathbb{R},X)$ with $\varphi_B^{\RR,\infty} = \frac{\pi}{2}$.
\end{satz}



We omit the proofs of Propositions~\ref{thm: Definitionsbereich Potenz der Zeitableitung} and~\ref{thm: H^infty Kalkül für Zeitableitung}. On the one hand this can be done very similar to the proofs of Propositions~\ref{thm: alpha Potenz von A} and~\ref{thm: H^infty Kalkül Laplace} respectively and on the other hand most of the assertions are already proved in~\cite{Denk-Saal-Seiler}.

\begin{lemma} \label{thm: Einbettung TLL}
Let $s \in \mathbb{R}$, $1 < p,q < \infty$ and $1 \le r \le \infty$ such that $p > \frac{n}{2}$.
Let $\delta > 0$ such that $\frac{n}{2p} + \delta < 1$.
Then we have the continuous embedding
\begin{equation*}
F^{s+2-\delta,r}_{p,q} \subset F^{s+1,r}_{2p,q}.
\end{equation*}
\end{lemma}

\begin{proof}
We use an embedding theorem for Triebel-Lizorkin spaces and deduce the result via interpolation.
For $\epsilon' := \frac{2}{n} (p - \frac{n}{2}) > 0$ we have
$\frac{1}{\delta} > 1 + \frac{1}{\epsilon'}$. Now select $0 < \epsilon < \epsilon'$ such that $\frac{1}{\delta} > 1 + \frac{1}{\epsilon}$. Then we have $\frac{n}{2} + \frac{n \epsilon}{2} < p$, so it's possible to select parameters
\begin{equation*}
\max \{ 1,\frac{n}{2} + \frac{n \epsilon}{2} \} < p_0 < p < p_1 < \infty.
\end{equation*}
Additionally, select $\theta \in (0,1)$ s.t. $\frac{1}{p} = \frac{1-\theta}{p_0} + \frac{\theta}{p_1}$.
From~\cite[Thm. 2.8.1]{Triebel1978} we get
$F^s_{p_j,q} \subset F^{s - \frac{n}{2 p_j}}_{2 p_j,q} \subset F^{s-1 + \delta}_{2 p_j,q}$ for $j=0,1$ and Theorem~\ref{thm: Interpolation von TLL-Räumen} gives
\begin{equation*}
F^{s+2-\delta,r}_{p,q}
= \Big( F^{s+2-\delta}_{p_0,q},F^{s+2-\delta}_{p_1,q} \Big)_{\theta,r}
\subset \Big( F^{s+1}_{2p_0,q},F^{s+1}_{2p_1,q} \Big)_{\theta,r}
= F^{s+1,r}_{2p,q}.
\qedhere
\end{equation*}
\end{proof}

\begin{lemma} \label{thm: Produkt}
Let $s > 0$, $1 < p,q < \infty$ and $1 \le r \le \infty$. Then the product \\ $\pi: F^{s,r}_{2p,q} \times F^{s,r}_{2p,q} \rightarrow F^{s,r}_{p,q}$ is continuous.
\end{lemma}

\begin{proof}
Again we make use of a corresponding fact for Triebel-Lizorkin spaces (which is included in the paper of J. Johnsen~\cite{Johnsen}) and extend this to Triebel-Lizorkin-Lorentz spaces via interpolation.
We fix parameters $1 < p_0 < p < p_1 < \infty$. The product
$\pi: F^s_{2p_j,q} \times F^s_{2p_j,q} \rightarrow F^s_{p_j,q}$
is continuous for $j=0,1$ due to~\cite[Thm. 6.1]{Johnsen} and so is
$\pi(\Punkt,u): F^s_{2p_j,q} \rightarrow F^s_{p_j,q}$ for each $u \in F^s_{2p_j,q}$.
Interpolation of the respective spaces leads to the continuity of
$\pi(\Punkt,u): F^{s,r}_{2p,q} \rightarrow F^{s,r}_{p,q}$
for any $u \in F^s_{2p_0,q} \cup F^s_{2p_1,q}$ and thus the whole product
$\pi: F^{s,r}_{2p,q} \times F^s_{p_i,q} \rightarrow F^{s,r}_{p,q}$
is continuous for $i=0,1$.

Now by repeating an analogue argument with
$\pi(v,\Punkt): F^s_{p_i,q} \rightarrow F^{s,r}_{2p,q}$
we get the continuity of
$\pi: F^{s,r}_{2p,q} \times F^{s,r}_{2p,q} \rightarrow F^{s,r}_{p,q}$, where we made use of the (simpler) fact, that we get $F^{s,r}_{p,q}$ by real interpolation with itself.
\end{proof}


Consider any function space $\mathcal{F}$ of time-dependent functions on
some time interval $(0,T)$ (or in other words on $[0,T]$ since we
usually identify two functions differing on a null set). If
$\mathcal{F}$ contains the smooth functions with compact support on
$(0,T]$ then we denote their closure in $\mathcal{F}$ by
$_0\mathcal{F}$. For the function spaces of time-dependent functions
that appear in the sequel, $_0\mathcal{F}$ consists of those functions $u \in \mathcal{F}$ with $u|_{t=0} = 0$ if the trace in time exists for $\mathcal{F}$. Note that we usually have $_0\mathcal{F} = \mathcal{F}$ if the trace in time doesn't exist (cf.~\cite[Thm. 4.3.2/1(a)]{Triebel1978}).

\begin{lemma} \label{thm: Mixed Derivatives}
Let $s \in \mathbb{R}$, $1 < p,q,r < \infty$, $1 < \eta < \infty$ and $\alpha \in [0,1]$. Then for $T \in (0,\infty]$ we have the continuous embeddings
\begin{equation} \label{eq: MDT auf R}
H^1_\eta \big( \mathbb{R},F^{s,r}_{p,q} \big) \cap L_\eta \big( \mathbb{R},F^{s+2,r}_{p,q} \big) \subset H^\alpha_\eta(\mathbb{R},F^{s+2(1-\alpha),r}_{p,q})
\end{equation}
and
\begin{equation} \label{eq: MDT auf (0,T)}
H^1_\eta \big( (0,T),F^{s,r}_{p,q} \big) \cap L_\eta \big( (0,T),F^{s+2,r}_{p,q} \big) \subset H^\alpha_\eta((0,T),F^{s+2(1-\alpha),r}_{p,q}).
\end{equation}
For $T \in (0,\infty)$ we also have the continuous embedding
\begin{equation} \label{eq: MDT mit Zeitspur 0}
_0H^1_\eta \big( (0,T),F^{s,r}_{p,q} \big) \cap L_\eta \big( (0,T),F^{s+2,r}_{p,q} \big) \subset \text{} _0H^\alpha_\eta((0,T),F^{s+2(1-\alpha),r}_{p,q})
\end{equation}
locally uniformly in time,
i.e., for every $T_0 > 0$ there exists an embedding constant $C > 0$ for~\eqref{eq: MDT mit Zeitspur 0}, which is independent of $T \in (0,T_0]$.
\end{lemma}

\begin{proof}
Let $\A = 1 - \Delta$ in $F^{s,r}_{p,q}$ be the operator from
Proposition~\ref{thm: alpha Potenz von A} and $B = 1 + \timeD$ in
$L_\eta(\mathbb{R},F^{s,r}_{p,q})$ the operator from
Proposition~\ref{thm: Definitionsbereich Potenz der Zeitableitung}. We
have already seen that $\A$ and $B$ have a bounded $H^\infty$-calculus
with $\varphi_\A^\infty + \varphi_B^\infty < \pi$. Note that $\A$ can be
interpreted as an operator in $L_\eta(\mathbb{R},F^{s,r}_{p,q})$ instead
of $F^{s,r}_{p,q}$ in a trivial way, where it still admits a bounded
$H^\infty$-calculus with the same angle $\varphi_\A^\infty = 0$.
Obviously $\A$ and $B$ are resolvent commuting operators. So all
conditions of the mixed derivative theorem (in the version of \cite[Lem.
4.1]{Denk-Saal-Seiler}) are fulfilled. This yields that
\begin{equation*}
\| \A^{1-\alpha} B^\alpha u \|_{L_\eta(\mathbb{R},F^{s,r}_{p,q})}
\le C \| \A u + Bu \|_{L_\eta(\mathbb{R},F^{s,r}_{p,q})}
\end{equation*}
holds for all $u \in \Def(\A) \cap \Def(B)$ and all $\alpha \in [0,1]$. Now we use Propositions~\ref{thm: Definitionsbereich Potenz der Zeitableitung}, \ref{thm: alpha Potenz von A} and~\ref{thm: Darstellung Ableitungen} and get for all $u \in \mathscr{S}(\mathbb{R},F^{s+2,r}_{p,q}) \subset H^1_\eta(\mathbb{R},F^{s,r}_{p,q}) \cap L_\eta(\mathbb{R},F^{s+2,r}_{p,q})$
\begin{multline*}
\| u \|_{H^\alpha_\eta(\mathbb{R},F^{s+2(1-\alpha),r}_{p,q})}
\sim \| B^\alpha u \|_{L_\eta(\mathbb{R},F^{s+2(1-\alpha),r}_{p,q})}
\sim \| B^\alpha u \|_{L_\eta(\mathbb{R},\Def(\A^{1-\alpha}))} \\
\sim \| \A^{1-\alpha} B^\alpha u \|_{L_\eta(\mathbb{R},F^{s,r}_{p,q})}
\lesssim \| \A u + Bu \|_{L_\eta(\mathbb{R},F^{s,r}_{p,q})}
\lesssim \| u \|_{H^1_\eta(\mathbb{R},F^{s,r}_{p,q}) \cap L_\eta(\mathbb{R},F^{s+2,r}_{p,q})}.
\end{multline*}
This proves~\eqref{eq: MDT auf R}.

We get~\eqref{eq: MDT auf (0,T)} as a conclusion of~\eqref{eq: MDT auf R} by suitable retraction and extension. More precise we make use of~\eqref{eq: Fortsetzungsoperator}, which yields an extension operator simultaneously on $H^1_\eta((0,T),F^{s,r}_{p,q})$ and on $L_\eta((0,T),F^{s+2,r}_{p,q})$.

In order to prove~\eqref{eq: MDT mit Zeitspur 0}, we make use of the extension operator~\eqref{eq: Fortsetzungsoperator Slobodecki mit Zeitspur 0} in the case $\beta = 1$.
For a fixed $T_0 > 0$ we get
\begin{equation*}
\begin{split}
\| u \|_{H^\alpha_\eta((0,T),F^{s+2(1-\alpha),r}_{p,q})}
& \le \| E_{\infty,1} E_T u \|_{H^\alpha_\eta(\mathbb{R},F^{s+2(1-\alpha),r}_{p,q})} \\
& \le C \| E_{\infty,1} E_T u \|_{H^1_\eta(\mathbb{R},F^{s,r}_{p,q}) \cap L_\eta(\mathbb{R},F^{s+2,r}_{p,q})} \\
& \le C' \| u \|_{H^1_\eta((0,T),F^{s,r}_{p,q}) \cap L_\eta((0,T),F^{s+2,r}_{p,q})}
\end{split}
\end{equation*}
for all $u \in$ $_0H^1_\eta \big( (0,T),F^{s,r}_{p,q} \big) \cap L_\eta \big( (0,T),F^{s+2,r}_{p,q} \big)$ with a constant $C' > 0$, independent of $T \in (0,T_0]$.
\end{proof}

We will additionally need the following embeddings for Bessel-potential spaces on a time-interval.

\begin{lemma} \label{thm: Soboleveinbettung}
Let $1 < \eta < \infty$ and let $X$ be a Banach space of class $\HT$. Then for $s > \frac{1}{2 \eta}$ and $T \in (0,\infty]$ we have the continuous embedding
\begin{equation} \label{eq: Einbettung Bessel-potential}
H^s_\eta((0,T),X) \subset L_{2 \eta}((0,T),X).
\end{equation}
For $\alpha > \frac{1}{2\eta}$ and $T_0 > 0$ the continuous embedding
\begin{equation} \label{eq: Einbettung Bessel-potential mit Zeitspur 0}
_0H^\alpha_\eta((0,T),X) \subset L_{2 \eta}((0,T),X)
\end{equation}
holds with an embedding constant $C > 0$, which is independent of $T \in (0,T_0]$.
\end{lemma}

\begin{proof}

For~\eqref{eq: Einbettung Bessel-potential mit Zeitspur 0} let first $\alpha \in (\frac{1}{\eta},1]$.
We select $\epsilon > 0$ such that $\alpha - 2 \epsilon > \frac{1}{\eta}$.
The embedding constant of $H^\alpha_\eta((0,T),X) \subset W^{\alpha - \epsilon}_\eta((0,T),X)$ doesn't depend on $T \in (0,\infty]$ and for the extension operator
\begin{equation*}
E_{\infty,1} E_T: \text{}_0W^{\alpha - \epsilon}_\eta((0,T),X)
\longrightarrow \text{}_0W^{\alpha - \epsilon}_\eta(\mathbb{R},X)
\end{equation*}
from~\eqref{eq: Fortsetzungsoperator Slobodecki mit Zeitspur 0}
there exists a continuity constant independent of $T \in (0,T_0]$.
Thus for $u \in$ $_0H^\alpha_\eta((0,T),X)$ we conclude
\begin{multline*}
\| u \|_{L_{2 \eta}((0,T),X)}
\le \| E_{\infty,1} E_T u \|_{L_{2 \eta}(\mathbb{R},X)}
\le C \| E_{\infty,1} E_T u \|_{H^{\alpha - 2 \epsilon}_\eta(\mathbb{R},X)} \\
\le C' \| E_{\infty,1} E_T u \|_{W^{\alpha - \epsilon}_\eta(\mathbb{R},X)}
\le C'' \| u \|_{W^{\alpha - \epsilon}_\eta((0,T),X)}
\le C''' \| u \|_{H^\alpha_\eta((0,T),X)},
\end{multline*}
where $C''' > 0$ is a constant independent of $T \in (0,T_0]$.
Now let $\alpha \in (\frac{1}{2 \eta},\frac{1}{\eta}]$. In this case we have $_0H^\alpha_\eta((0,T),X) = H^\alpha_\eta((0,T),X)$,
(see~\cite[Thm. 4.3.2/1(a)]{Triebel1978})
so we can make use of an extension argument as well, where we have the trivial extension available this time.
The case $\alpha > 1$ is an obvious consequence.

Relation \eqref{eq: Einbettung Bessel-potential} is a well-known Sobolev embedding.
It can be obtained by an analogous extension argument as above, where we make use of~\eqref{eq: Fortsetzungsoperator Slobodecki} instead of~\eqref{eq: Fortsetzungsoperator Slobodecki mit Zeitspur 0}.
For $\mathbb{R}$ instead of $(0,T)$ see e.g.~\cite[Thm. 3.7.5]{Amann09}.
\end{proof}

\section{The Navier-Stokes equations} \label{sec: nse}

We fix $s \in \mathbb{R}$, $1 < p,q,r < \infty$, $1 < \eta < \infty$ and $X_\sigma := (F^{s,r}_{p,q})^n_\sigma$ with dimension $n \ge 2$.
As above, $A_S$ is the Stokes operator in $X_\sigma$.
The solution space for the Stokes equation is
\begin{equation*}
\mathbb{E}_T := H^1_\eta \big( (0,T),X_\sigma \big) \cap L_\eta \big( (0,T),\Def(A_S) \big),
\end{equation*}
where $T \in (0,\infty]$.
Next, as in Section~\ref{sec: Basic}, we set
\begin{equation}
\mathbb{F}_T := L_\eta((0,T),X_\sigma)
\quad \text{and} \quad
\mathbb{I} := \{ u_0 = u(0) : u \in \mathbb{E}_T
\},
\end{equation}
equipped with the norm
$\| u_0 \|_\mathbb{I} = \inf_{u(0) = u_0} \| u \|_{\mathbb{E}_T}$,
so $\mathbb{F}_T \times \mathbb{I}$ is the data space with right-hand side functions $f \in \mathbb{F}_T$ and initial values $u_0 \in \mathbb{I}$.
Note that by (\ref{charini}), Proposition~\ref{thm: hhd}, and~\cite[Thm.~1.9.3/1]{Triebel1978} we obtain
\[
	\mathbb{I}=\left(X_\sigma,\Def(A_S)\right)_{1-1/\eta,\eta}
	=P\left(F^{s,r}_{p,q},\, F^{s+2,r}_{p,q}\right)^n_{1-1/\eta,\eta}.
\]

The solution operator for the Stokes equation,
\begin{equation}
L: \mathbb{E}_T \xrightarrow{~\cong~} \mathbb{F}_T \times \mathbb{I}, \quad
u \longmapsto
\begin{pmatrix}
(\timeD - A_S)u \\
u(0)
\end{pmatrix},
\end{equation}
is an isomorphism when $T < \infty$, due to Proposition~\ref{thm: Stokesoperator MR}.
The nonlinear term is
\begin{equation*}
G(u):= - P (u \cdot \nabla) u = - P \mathrm{div} (u u^T), \quad u \in (F^{s,r}_{p,q})^n_\sigma,
\end{equation*}
where $P \in \mathscr{L}((F^{s,r}_{p,q})^n)$ denotes the Helmholtz projection introduced in Section~\ref{sec: Helmholtz Projection and Stokes Operator}.

\begin{theorem} \label{thm: Hauptsatz}
Let $n \in \mathbb{N}$, $n \ge 2$, $s > -1$ and let $1 < p,q,r < \infty$ and $1 < \eta < \infty$ such that $\frac{n}{2p} + \frac{1}{\eta} < 1$. Then for all
$\left( \begin{smallmatrix} f \\ u_0 \end{smallmatrix} \right)
\in \mathbb{F}_\infty \times \mathbb{I}$
\begin{equation*}
\mathrm{(PNSE)}_{f,u_0}
\begin{cases}
\timeD u - \Delta u + P(u \cdot \nabla)u &= ~f \quad ~\text{in } (0,T) \times \mathbb{R}^n, \\
 \qquad \qquad \qquad \qquad u(0) &= ~u_0 \quad \text{in } \mathbb{R}^n
\end{cases}
\end{equation*}
has a unique maximal strong solution
$u: [0,T^*) \longrightarrow \mathbb{I}$ with $T^* \in (0,\infty]$ and $u \in \mathbb{E}_T$ for all $T \in (0,T^*)$.
If additionally $\frac{n}{2p} + \frac{2}{\eta} < 1$, then $u$ is either a global solution or we have $T^* < \infty$ and $\mathrm{limsup}_{t \nearrow T^*} \| u(t) \|_{\mathbb{I}} = \infty$.
\end{theorem}

We first convince ourselves that the two systems 
$(NSE)_{f,u_0}$ and $(PNSE)_{f,u_0}$ are equivalent. This particularly
shows that
Theorem~\ref{thm: Hauptsatz} implies Theorem~\ref{thm: Hauptsatz0}.
Indeed, when $u$ is the solution of $\mathrm{(PNSE)}_{f,u_0}$
given by Theorem~\ref{thm: Hauptsatz}, we get the solution $(u,\nabla p)$ of
$\mathrm{(NSE)}_{f,u_0}$ as claimed in Theorem~\ref{thm: Hauptsatz0}
by setting $\nabla p = -(1 - P)(u \cdot \nabla)u$. 
On the other hand, if $(u,\nabla p)$ is a solution of $\mathrm{(NSE)}_{f,u_0}$, 
then $u$ solves $\mathrm{(PNSE)}_{f,u_0}$ and consequently 
$\nabla p = -(1 - P)(u \cdot \nabla)u$.

The proof of the additional statement in Theorem~\ref{thm: Hauptsatz} 
will essentially make use of the following embedding for the 
space of initial values.

\begin{lemma} \label{thm: Einbettung Raum der Anfangswerte}
Let $s \in \mathbb{R}$, $1 < p,q,r < \infty$ and $1 < \eta < \infty$ such that $\frac{n}{2p} + \frac{2}{\eta} < 1$. Then we have the continuous embedding
\begin{equation*}
\mathbb{I} \subset (F^{s+1,r}_{2p,q})^n.
\end{equation*}
\end{lemma}

\begin{proof}
Select $0 < \epsilon < \min \{\eta-1, \frac{\eta}{2} [1 - (\frac{n}{2p} + \frac{2}{\eta})] \}$ and $T \in (0,\infty)$. Then we have the continuous embeddings
\begin{equation*}
\mathbb{E}_T
\subset H^{\frac{1+\epsilon}{\eta}}_\eta \big( (0,T),\big( F^{s+2(1-\frac{1+\epsilon}{\eta}),r}_{p,q} \big)^n \big)
\subset C \big( [0,T],\big( F^{s+2(1-\frac{1+\epsilon}{\eta}),r}_{p,q} \big)^n \big),
\end{equation*}
where the first embedding follows from Lemma~\ref{thm: Mixed Derivatives} and the second one can be deduced from standard Sobolev embedding in the same way as in the proof of Lemma~\ref{thm: Soboleveinbettung}.
Now, setting $\delta := \frac{2(1+\epsilon)}{\eta}$, we get $\frac{n}{2p} + \delta < 1$ so Lemma~\ref{thm: Einbettung TLL} gives the continuous embedding
\begin{equation*}
\big(F^{s+2(1-\frac{1+\epsilon}{\eta}),r}_{p,q}\big)^n
\subset (F^{s+1,r}_{2p,q})^n.
\end{equation*}
This leads to $\| u_0 \|_{(F^{s+1,r}_{2p,q})^n} \le C \| u \|_{\mathbb{E}_T}$
for $u_0 \in \mathbb{I}$ and any $u \in \mathbb{E}_T$ with $u(0) = u_0$ so the assertion is proved.
\end{proof}

\begin{proof}[Proof of Theorem~\ref{thm: Hauptsatz}]
Let $\left( \begin{smallmatrix} f \\ u_0 \end{smallmatrix} \right)
\in \mathbb{F}_\infty \times \mathbb{I}$.
We start with the local existence and uniqueness, so we need to show that there is a unique solution $u \in \mathbb{E}_T$ for
\begin{equation*}
Lu = \begin{pmatrix} f + G(u) \\ u_0 \end{pmatrix}
\end{equation*}
on some time interval.
First of all we note that it's possible to restrict ourselves to those solutions with $u(0) = 0$. In fact, by setting $u^* := L^{-1} \left( \begin{smallmatrix} f \\ u_0 \end{smallmatrix} \right)$, we can always consider $\widebar{u} = u - u^* \in$ $_0\mathbb{E}_T$ for $u \in \mathbb{E}_T$, so for any $T \in (0,\infty)$ the following assertions are equivalent:
\begin{enumerate}[(a)]
\item \label{eq: eindeutige Lösung 1} $Lu = \begin{pmatrix} f + G(u) \\ u_0 \end{pmatrix}$ has a unique solution $u \in \mathbb{E}_T$.
\item \label{eq: eindeutige Lösung 2} $L \widebar{u} = \begin{pmatrix} G(\widebar{u} + u^*) \\ 0 \end{pmatrix}$ has a unique solution $\widebar{u} \in$ $_0\mathbb{E}_T$.
\end{enumerate}

Before we are able to verify~\eqref{eq: eindeutige Lösung 2}, it's necessary to have the continuous embedding
\begin{equation} \label{eq: Einbettung nichtlinearer Term}
G(\mathbb{E}_T) \subset \mathbb{F}_T
\end{equation}
for $T \in (0,\infty)$. For $u \in \mathbb{E}_T$, using Proposition~\ref{thm: Darstellung Ableitungen}, we have
\begin{multline*}
\| G(u) \|_{\mathbb{F}_T}
=   \| P(u \cdot \nabla)u \|_{L_\eta((0,T),(F^{s,r}_{p,q})^n)}
\le C \| \mathrm{div} (uu^T) \|_{L_\eta((0,T),(F^{s,r}_{p,q})^n)} \\
\le C' \| uu^T \|_{L_\eta((0,T),(F^{s+1,r}_{p,q})^{n \times n})}
\le C'' \| u \|_{L_{2 \eta}((0,T),(F^{s+1,r}_{2p,q})^n)}^2,
\end{multline*}
where we applied Hölder's inequality together with Lemma~\ref{thm: Produkt} (note that $s+1 > 0$ is assumed) to get the last inequality. Now it remains to prove $\mathbb{E}_T \subset L_{2 \eta}((0,T),(F^{s+1,r}_{2p,q})^n)$, to get~\eqref{eq: Einbettung nichtlinearer Term}.
Due to the condition $\frac{n}{2p} + \frac{1}{\eta} < 1$, we can select $\delta > \frac{1}{\eta}$ such that $\frac{n}{2p} + \delta < 1$. Then we have $F^{s+2-\delta,r}_{p,q} \subset F^{s+1,r}_{2p,q}$, according to Lemma~\ref{thm: Einbettung TLL}. By setting $\alpha := \frac{\delta}{2}$ we get the continuous embeddings
\begin{equation} \label{eq: Einbettung ohne Nullspur}
\begin{split}
\mathbb{E}_T
\subset H^\alpha_\eta \big( (0,T),(F^{s+2(1-\alpha),r}_{p,q})^n \big)
&\subset L_{2 \eta} \big( (0,T),(F^{s+2(1-\alpha),r}_{p,q})^n \big) \\
&\subset L_{2 \eta} \big( (0,T),(F^{s+1,r}_{2p,q})^n \big)
\end{split}
\end{equation}
where we used Lemma~\ref{thm: Mixed Derivatives} for the first embedding, Lemma~\ref{thm: Soboleveinbettung} for the second embedding and Lemma~\ref{thm: Einbettung TLL} for the last embedding.
This yields~\eqref{eq: Einbettung nichtlinearer Term}.

In order to obtain~\eqref{eq: eindeutige Lösung 2}, we define
\begin{equation}
N: \text{}_0\mathbb{E}_T \longrightarrow \mathbb{F}_T \times \{ 0 \}, \quad
\widebar{u} \longmapsto L \widebar{u} - \begin{pmatrix} G(\widebar{u} + u^*) \\ 0 \end{pmatrix}
\end{equation}
for $T \in (0,\infty)$.
Because of~\eqref{eq: Einbettung nichtlinearer Term} we know that $N$ is well-defined, i.e., we have indeed $N (\widebar{u}) \in \mathbb{F}_T \times \{ 0 \}$ for all $\widebar{u} \in \text{}_0\mathbb{E}_T$. Furthermore, $N$ is continuously Fréchet-differentiable, where
\begin{equation*}
DN(0)v = Lv - \begin{pmatrix} DG(u^*)v \\ 0 \end{pmatrix}
= Lv + \begin{pmatrix} P(u^* \cdot \nabla)v + P(v \cdot \nabla)u^* \\ 0 \end{pmatrix}
\quad \forall v \in \text{}_0\mathbb{E}_T
\end{equation*}
is the derivative at the zero point. Our aim is to verify that there exists a unique $\widebar{u} \in \text{}_0\mathbb{E}_T$ such that $N(\widebar{u}) = 0$ for small time intervals $(0,T)$.

As a first step to obtain this, we prove that $DN(0):$ $_0\mathbb{E}_T
\rightarrow \mathbb{F}_T \times \{ 0 \}$ is an isomorphism when $T > 0$
is small enough. Similarly to the verification of~\eqref{eq: Einbettung nichtlinearer Term} we obtain for $T > 0$ and $v \in$ $_0\mathbb{E}_T$ that
\begin{align} \label{eq: Abschaetzung Nichtlinearitaet}
\Big\| \begin{pmatrix} DG(u^*)v \\ 0 \end{pmatrix} \Big\|_{\mathbb{F}_T \times \{ 0 \}}
&= \| P \mathrm{div}(u^* v^T) + P \mathrm{div}(v (u^*)^T) \|_{L_\eta((0,T),(F^{s,r}_{p,q})^n)} \notag \\
&\le C \| \mathrm{div}(u^* v^T) + \mathrm{div}(v (u^*)^T) \|_{L_\eta((0,T),(F^{s,r}_{p,q})^n)} \notag \\
&\le C' \| u^* v^T + v (u^*)^T \|_{L_\eta((0,T),(F^{s+1,r}_{p,q})^{n \times n})} \notag \\
&\le C'' \| u^* \|_{L_{2 \eta}((0,T),(F^{s+1,r}_{2p,q})^n)}
         \| v \|_{L_{2 \eta}((0,T),(F^{s+1,r}_{2p,q})^n)},
\end{align}
in view of Proposition~\ref{thm: Darstellung Ableitungen}, Lemma~\ref{thm: Produkt} and Hölder's inequality, where the constant $C'' > 0$ is independent of $T \in (0,\infty)$.
Again let $\delta > \frac{1}{\eta}$ such that $\frac{n}{2p} + \delta < 1$ and set $\alpha := \frac{\delta}{2}$. Then we have $F^{s+2-\delta,r}_{p,q} \subset F^{s+1,r}_{2p,q}$ and, since $\alpha > \frac{1}{2 \eta}$, we obtain for any fixed $T_0 > 0$
\begin{equation} \label{eq: Einbettung mit Nullspur}
\begin{split}
_0\mathbb{E}_T
\subset \text{}_0H^\alpha_\eta \big( (0,T),(F^{s+2(1-\alpha),r}_{p,q})^n \big)
&\subset L_{2 \eta} \big( (0,T),(F^{s+2(1-\alpha),r}_{p,q})^n \big) \\
&\subset L_{2 \eta} \big( (0,T),(F^{s+1,r}_{2p,q})^n \big)
\end{split}
\end{equation}
where the embeddings are continuous with an embedding constant independent of $T \in (0,T_0]$, due to Lemmas~\ref{thm: Mixed Derivatives} and~\ref{thm: Soboleveinbettung}.
Hence we have in total
\begin{equation} \label{eq: Abschätzung Störung}
\Big\| \begin{pmatrix} DG(u^*)v \\ 0 \end{pmatrix} \Big\|_{\mathbb{F}_T \times \{ 0 \}}
\le C_1 \| u^* \|_{L_{2 \eta}((0,T),(F^{s+1,r}_{2p,q})^n)}
         \| v \|_{\mathbb{E}_T}
\end{equation}
for all $v \in$ $_0\mathbb{E}_T$ and for all $T \in (0,T_0]$.
Thanks to Lemma~\ref{thm: Abschätzung MR} there is also a constant $C_2 > 0$ such that $\| L^{-1} \|_{\mathscr{L}(\mathbb{F}_T \times \{ 0 \}, \text{}_0\mathbb{E}_T)} \le C_2$ for all $T \in (0,T_0]$.

The size of the finite time interval $(0,T_0)$ was arbitrary up to this point. Proceeding from any finite $T_0 > 0$, we will shrink the interval $(0,T_0)$ in the following to get a unique local solution. The constants $C_1$ and $C_2$, found above, can be assumed to be fixed so they don't change by shrinking $(0,T_0)$.
First, let $(0,T_0)$ be small enough, so that
\begin{equation} \label{eq: Abschätzung u*}
\| u^* \|_{L_{2 \eta}((0,T_0),(F^{s+1,r}_{2p,q})^n)} \le \frac{1}{2 C_1 C_2}
\end{equation}
holds.
Then we obtain from~\eqref{eq: Abschätzung Störung} and~\eqref{eq: Abschätzung u*}
\begin{equation*}
\Big\| \begin{pmatrix} DG(u^*) \\ 0 \end{pmatrix} \Big\|_{\mathscr{L}(_0\mathbb{E}_T,\mathbb{F}_T \times \{ 0 \})}
<   \frac{1}{\| L^{-1} \|_{\mathscr{L}(\mathbb{F}_T \times \{ 0 \}, \text{}_0\mathbb{E}_T)}}
\end{equation*}
for all $T \in (0,T_0]$. Hence by the Neumann series we get that $DN(0):$ $_0\mathbb{E}_T \rightarrow \mathbb{F}_T \times \{ 0 \}$ is an isomorphism for all $T \in (0,T_0]$.

We apply the inverse function theorem (see e.g.~\cite[Thm. VII.7.3]{Amann-Escher-2}) and get open neighborhoods $0 \in U_T \subset$ $_0\mathbb{E}_T$ and $N(0) \in V_T \subset$ $\mathbb{F}_T \times \{ 0 \}$ such that $N: U_T \rightarrow V_T$ is bijective. Now we fix $T \in (0,T_0]$ and define for $0 < T' < T$ a function $F_{T'} \in \mathbb{F}_T$ by
\begin{equation*}
F_{T'}(t) :=
\begin{cases}
0, & \text{if } t \in (0,T') \\
G(u^*)(t), & \text{if } t \in [T',T).
\end{cases}
\end{equation*}
Then we have
\begin{equation*}
\begin{split}
\Big\| \begin{pmatrix} F_{T'} \\ 0 \end{pmatrix} - \begin{pmatrix} G(u^*) \\ 0 \end{pmatrix} \Big\|_{\mathbb{F}_T \times \{ 0 \}}^\eta
= \int_0^T \| F_{T'}(t) - G(u^*)(t) \|_{X_\sigma}^\eta dt \\
= \int_0^{T'} \| G(u^*)(t) \|_{X_\sigma}^\eta dt
\xrightarrow{T' \searrow 0} 0
\end{split}
\end{equation*}
and thus
$\begin{pmatrix} F_{T'} \\ 0 \end{pmatrix}
\xrightarrow{T' \searrow 0} N(0)$ in $\mathbb{F}_T \times \{ 0 \}$. Since $V_T$ is a neighborhood of $N(0)$, this yields $\begin{pmatrix} F_{T'} \\ 0 \end{pmatrix} \in V_T$, if $T' \in (0,T)$ is small enough and consequently for $\widebar{u} := N^{-1} \begin{pmatrix} F_{T'} \\ 0 \end{pmatrix} \in U_T$ we have
$N(\widebar{u}) = \begin{pmatrix} F_{T'} \\ 0 \end{pmatrix} = \begin{pmatrix} 0 \\ 0 \end{pmatrix}$ on $(0,T')$. Hence, by restriction of $\widebar{u}$ to $(0,T')$, we get a solution $\widebar{u} \in$ $_0\mathbb{E}_{T'}$ of~\eqref{eq: eindeutige Lösung 2}. Since $N: U_{T} \rightarrow V_{T}$ is bijective, this solution is unique.

Having established the local existence and uniqueness of a solution for $\mathrm{(PNSE)}_{f,u_0}$, we now extend the solution to a maximal time interval $[0,T^*)$. First we note that uniqueness holds on any time interval: Considering two solutions $u,v \in \mathbb{E}_T$ of $\mathrm{(PNSE)}_{f,u_0}$ on $[0,T)$ for some $T \in (0,\infty]$, we know from the established local uniqueness that $u = v$ holds on some $[0,T') \subset [0,T)$. We assume that $u$ and $v$ do not coincide on $[0,T)$. Then Lemma~\ref{thm: Einbettung Phasenraum} allows to apply a continuity argument, which provides some $0 < t_1 < t_2 < T$ so that
$u(t) = v(t) \text{ for all } t \in [0,t_1] \text{ and } u(t) \ne v(t) \text{ for all } t \in (t_1,t_2)$. Now, setting $u_1 := u(t_1)$ and $f_1 := f(t_1 + \cdot)$, we can apply local uniqueness of the solution of $\mathrm{(PNSE)}_{f_1,u_1}$ and get $u(t_1 + \cdot) = v(t_1 + \cdot)$ on some $[0,T'')$, contradictory to $u(t) \ne v(t) \text{ for all } t \in (t_1,t_2)$.

In order to get a maximal time interval $[0,T^*)$ for the solution of $\mathrm{(PNSE)}_{f,u_0}$, we define
\begin{equation*}
\begin{split}
& M := \big\{ (J_T,u_T) : T \in (0,\infty), \exists \text{ solution } u_T \in \mathbb{E}_T \text{ of } \mathrm{(PNSE)}_{f,u_0} \text{ on } J_T = [0,T) \big\}, \\
& J^* := \bigcup \{ J_T : (J_T,u_T) \in M \} =: [0,T^*)
\end{split}
\end{equation*}
and $u: [0,T^*) \rightarrow \mathbb{I}$, $u(t) := u_T(t)$ for $t \in J_T$. Due to the uniqueness proved above, $u$ is well defined and consequently the desired maximal solution.

Now let additionally $\frac{n}{2p} + \frac{2}{\eta} < 1$. We assume $T^* < \infty$ and $\mathrm{limsup}_{t \nearrow T^*} \| u(t) \|_{\mathbb{I}} < \infty$ for the maximal solution $u$. Then we have $u \in BC([0,T^*),\mathbb{I})$ (i.e., bounded and continuous).
For $T \in (0,T^*]$ and $v \in \mathbb{E}_T$ we define the linear operator
\begin{equation*}
Bv :=
\begin{pmatrix}
P \mathrm{div} (u v^T) \\
0
\end{pmatrix}.
\end{equation*}
Then we have $(L + B)u = \big( \begin{smallmatrix} f \\ u_0 \end{smallmatrix} \big)$.
As in~\eqref{eq: Abschaetzung Nichtlinearitaet} we get
\begin{equation} \label{eq: Abschaetzung fuer Stoerung}
\| Bv \|_{\mathbb{F}_T \times \mathbb{I}}
\le C \| u \|_{L_{2 \eta}((0,T),(F^{s+1,r}_{2p,q})^n)} \| v \|_{L_{2 \eta}((0,T),(F^{s+1,r}_{2p,q})^n)}
\quad \forall v \in \mathbb{E}_T
\end{equation}
with a constant $C > 0$ independent of $T$.
Concerning~\eqref{eq: Einbettung mit Nullspur} and Lemma~\ref{thm: Einbettung Raum der Anfangswerte} we get
\begin{equation*}
\| Bv \|_{\mathbb{F}_T \times \mathbb{I}}
\le C' \Big( \int_0^T \| u(t) \|_{(F^{s+1,r}_{2p,q})^n)}^{2 \eta} \Big)^\frac{1}{2 \eta} \| v \|_{\mathbb{E}_T}
\le C'' T^\frac{1}{2 \eta} \| u \|_{BC([0,T^*),\mathbb{I})} \| v \|_{\mathbb{E}_T}
\end{equation*}
for all $v \in \text{}_0\mathbb{E}_T$ with a constant $C'' > 0$ independent of $T \in (0,T^*]$.
Due to~\eqref{eq: Einbettung ohne Nullspur} we can also deduce
$B \in \mathscr{L}(\mathbb{E}_T,\mathbb{F}_T \times \mathbb{I})$ from~\eqref{eq: Abschaetzung fuer Stoerung}.
Furthermore, Lemma~\ref{thm: Abschätzung MR} gives a constant $K > 0$, such that
$\| L^{-1} \|_{\mathscr{L}(\mathbb{F}_T \times \{ 0 \},\text{}_0\mathbb{E}_T)} \le K$ holds for all $T \in (0,T^*]$.
Consequently, we obtain for sufficiently small $T \in (0,T^*]$ that
\begin{equation*}
\| B \|_{\mathscr{L}(\text{}_0\mathbb{E}_T,\mathbb{F}_T \times \{ 0 \})}
< \frac{1}{\| L^{-1} \|_{\mathscr{L}(\mathbb{F}_T \times \{ 0 \},\text{}_0\mathbb{E}_T)}},
\end{equation*}
which yields that $L + B: \text{}_0\mathbb{E}_T \rightarrow \mathbb{F}_T \times \{ 0 \}$ is an isomorphism.
More precisely, we need to choose
\begin{equation} \label{eq: Wahl von T}
T \le \frac{1}{(2 C'' K \| u \|_{BC([0,T^*),\mathbb{I})})^{2 \eta}}.
\end{equation}
Now, for $T$ as in~\eqref{eq: Wahl von T}, we can select $T_1 \in (0,T^*)$ and repeat the argument on $(T_1,T + T_1)$ instead of $(0,T)$. This yields that $L + B: \text{}_0\mathbb{E}_{(T_1,T + T_1)} \rightarrow \mathbb{F}_{(T_1,T + T_1)} \times \{ 0 \}$ is an isomorphism, where $\text{}_0\mathbb{E}_{(T_1,T + T_1)}$ (resp.\ $\mathbb{F}_{(T_1,T + T_1)}$) consists of the translations of functions in $\text{}_0\mathbb{E}_T$ (resp.\ $\mathbb{F}_T$) by $T_1$. We repeat this argument $k$ times on the interval $(k T_1, T + k T_1) \cap (0,T^*)$ until we reach $T + k T_1 \ge T^*$.
Finally we have that $L + B: \text{}_0\mathbb{E}_{T^*} \rightarrow \mathbb{F}_{T^*} \times \{ 0 \}$ is an isomorphism.
Now it is not hard to deduce that
\begin{equation} \label{eq: Isomorphismus}
L + B: \mathbb{E}_{T^*} \overset{\cong}{\longrightarrow} \mathbb{F}_{T^*} \times \mathbb{I}
\end{equation}
is an isomorphism: Continuity and injectivity are obvious while one gets the surjectivity by setting
$v^* := L^{-1} \big( \begin{smallmatrix} 0 \\ v_0 \end{smallmatrix} \big)$ and
$v := (L+B)^{-1} \big( \begin{smallmatrix} g - P \mathrm{div} (v^* u^T) \\ 0 \end{smallmatrix} \big) + v^* \in \mathbb{E}_{T^*}$ for $ \big( \begin{smallmatrix} g \\ v_0 \end{smallmatrix} \big) \in \mathbb{F}_T \times \mathbb{I}$.
As a consequence of~\eqref{eq: Isomorphismus} and Lemma~\ref{thm:
Einbettung Phasenraum} we can achieve
\begin{equation*}
u = (L + B)^{-1} \begin{pmatrix} f \\ u_0 \end{pmatrix}
\in \mathbb{E}_{T^*} \subset BUC([0,T^*),\mathbb{I})
\end{equation*}
and hence $u(T^*) = \lim_{t \nearrow T^*} u(t) \in \mathbb{I}$.
Application of the local existence and uniqueness now gives a solution of
$\mathrm{(PNSE)}_{f(\cdot + T^*),u(T^*)}$ on some time interval $[0,T'')$, which yields an extension of $u$ to a solution of $\mathrm{(PNSE)}_{f,u_0}$ on $[0,T^*+T'')$, in contradiction to the maximality of $u$.
\end{proof}

\section{Appendix: Extension operators} \label{sec: appendix}

Let $1 < \eta < \infty$.
For fixed $m \in \mathbb{N}$ and $T \in (0,\infty]$ there exists a mapping $u \mapsto E_{T,m} u$ for functions $u$ (defined on $(0,T)$ with values in any vector space) such that for all $k \in \{ 0,1,\dots,m \}$ and any Banach space $X$ we have an extension operator
\begin{equation} \label{eq: Fortsetzungsoperator}
E_{T,m}: W^k_\eta((0,T),X) \longrightarrow W^k_\eta(\mathbb{R},X).
\end{equation}
A precise proof can be found in~\cite[Thm. 4.26]{Adams} for the case of scalar-valued functions, but the given proof can be directly transferred to the vector-valued case. $E_{T,m}$ is the coretraction of
\begin{equation*}
R: W^k_\eta(\mathbb{R},X)
\longrightarrow
W^k_\eta((0,T),X), \quad
u \longmapsto u|_{(0,T)},
\end{equation*}
so, by the interpolation
$W^s_\eta(\mathbb{R},X)
= \big( L_\eta(\mathbb{R},X) , W^k_\eta(\mathbb{R},X) \big)_{\frac{s}{k},\eta}$
for $0 < s < k$, we get
\begin{equation*}
W^s_\eta((0,T),X)
= \big( L_\eta((0,T),X) , W^k_\eta((0,T),X) \big)_{\frac{s}{k},\eta}
\end{equation*}
and the extension operator
\begin{equation} \label{eq: Fortsetzungsoperator Slobodecki}
E_{T,m}: W^s_\eta((0,T),X) \longrightarrow W^s_\eta(\mathbb{R},X)
\end{equation}
(see~\cite[Thm. 1.2.4]{Triebel1978}).

Now let $T \in (0,\infty)$, $1 < \eta < \infty$ and $X$ a Banach space. For a function $u$ defined on $(0,T)$ with values in any vector space we set
\begin{equation*}
E_T u(\tau) :=
\begin{cases}
u(\tau), & \text{if } 0 <  \tau < T \\
u(2T - \tau), & \text{if } T \le \tau < 2T \\
0, & \text{if } 2T \le \tau
\end{cases}
\end{equation*}
(see also~\cite{Pruess-Saal-Simonett}).
Then, due to~\cite[Prop. 6.1]{Pruess-Saal-Simonett}, this leads to an extension operator
\begin{equation}
E_T: \text{}_0W^\beta_\eta \big( (0,T),X \big) \longrightarrow \text{}_0W^\beta_\eta \big( (0,\infty),X \big)
\end{equation}
for $\beta \in (\frac{1}{p},1]$ such that for any fixed $T_0 \in (0,\infty)$ there is a constant $C = C(T_0)$ with $\| E_T \| \le C$ for all $T \in (0,T_0]$.
Now we use~\eqref{eq: Fortsetzungsoperator Slobodecki} in the case $T = \infty$ and $m = 1$ and get the extension operator
\begin{equation} \label{eq: Fortsetzungsoperator Slobodecki mit Zeitspur 0}
E_{\infty,1} E_T: \text{}_0W^\beta_\eta \big( (0,T),X \big) \longrightarrow \text{}_0W^\beta_\eta \big( \mathbb{R},X \big)
\end{equation}
(for $\beta \in (\frac{1}{p},1]$),
whose operator norms $\| E_{\infty,1} E_T \|$, $T \in (0,T_0]$ are bounded above for a fixed $T_0 > 0$ as well.
The structure of $E_T$ also gives that $\| E_T u \|_{L_\eta((0,\infty),X)} \le 2 \| u \|_{L_\eta((0,T),X)}$.

\addcontentsline{toc}{section}{Literatur}


\begin{thebibliography}{10}

\bibitem{Adams}
R.~Adams.
\newblock {\em Sobolev Spaces}.
\newblock Academic Press, New York, 1975.

\bibitem{Amann97}
H.~Amann.
\newblock Operator-valued {F}ourier multipliers, vector-valued {B}esov spaces,
  and applications.
\newblock {\em Math. Nachr.}, 186:5--56, 1997.

\bibitem{Amann_Quasilinear}
H.~Amann.
\newblock {\em Linear and Quasilinear Parabolic Problems. {V}ol. {I}}.
\newblock Birkhäuser, Basel - Boston - Berlin, 2008.

\bibitem{Amann09}
H.~Amann.
\newblock {\em Anisotropic function spaces and maximal regularity for parabolic
  problems. {P}art 1}.
\newblock Matfyzpress, Prague, 2009.

\bibitem{Amann-Escher-2}
H.~Amann and J.~Escher.
\newblock {\em Analysis 2}.
\newblock Birkhäuser, Basel - Boston - Berlin, 2008.

\bibitem{Denk-Hieber-Pruess}
R.~Denk, M.~Hieber, and J.~Prüss.
\newblock {$\mathcal R$}-boundedness, {F}ourier multipliers and problems of
  elliptic and parabolic type.
\newblock {\em Mem. Amer. Math. Soc.}, 166:viii+114, 2003.

\bibitem{Denk-Saal-Seiler}
R.~Denk, J.~Saal, and J.~Seiler.
\newblock Inhomogeneous symbols, the {N}ewton polygon, and maximal
  {$L^p$}-regularity.
\newblock {\em Russ. J. Math. Phys.}, 15:171--191, 2008.

\bibitem{Galdi}
G.~P. Galdi.
\newblock {\em An Introduction to the Mathematical Theory of the Navier-Stokes
  Equations}.
\newblock Springer, New York - Dordrecht - Heidelberg - London, 2011.

\bibitem{Gra1_3rd}
L.~Grafakos.
\newblock {\em Classical Fourier Analysis}.
\newblock Springer, New York, third edition, 2014.

\bibitem{Haase}
M.~Haase.
\newblock {\em The functional calculus for sectorial operators}.
\newblock Birkhäuser, Basel, 2006.

\bibitem{hisa2016}
M.~Hieber and J.~Saal.
\newblock {\em The Stokes Equation in the {$L^p$}-setting: Well Posedness and
  Regularity Properties}, pages 1--88.
\newblock Handbook of Mathematical Analysis in Mechanics of Viscous Fluids.
  Springer, 2016.

\bibitem{Hytoenen}
T.~Hytönen, J.~van Neerven, M.~Veraar, and L.~Weis.
\newblock {\em Analysis in Banach Spaces}.
\newblock Springer, 2016.

\bibitem{Johnsen}
J.~Johnsen.
\newblock Pointwise multiplication of {B}esov and {T}riebel-{L}izorkin spaces.
\newblock {\em Math. Nachr.}, 175:85--133, 1995.

\bibitem{Kaip}
M.~Kaip and J.~Saal.
\newblock The permanence of $\mathcal{R}$-boundedness and property{$(\alpha)$}
  under interpolation and applications to parabolic systems.
\newblock {\em J. Math. Sci. Univ. Tokyo}, 19(3):359--407, 2012.

\bibitem{Kunstmann-Weis}
P.~Kunstmann and L.~Weis.
\newblock {Maximal $L_p$-regularity for parabolic equations, Fourier multiplier
  theorems and $H^\infty$- functional calculus.}
\newblock In {\em {Functional analytic methods for evolution equations. Based
  on lectures given at the autumn school on evolution equations and semigroups,
  Levico Terme, Trento, Italy, October 28--November 2, 2001}}, pages 65--311.
  Springer, Berlin, 2004.

\bibitem{Pruess-Saal-Simonett}
J.~Prüss, J.~Saal, and G.~Simonett.
\newblock Existence of analytic solutions for the classical {S}tefan problem.
\newblock {\em Math. Ann.}, 338:703--755, 2007.

\bibitem{Pruess-Simonett}
J.~Prüss and G.~Simonett.
\newblock {\em Moving Interfaces and Quasilinear Parabolic Evolution
  Equations}.
\newblock Birkhäuser, 2016.

\bibitem{RubioDeFrancia}
J.~Rubio~de Francia.
\newblock Martingale and integral transforms of {B}anach space valued
  functions.
\newblock In {\em Probability and Banach Spaces}, pages 195--222. Springer,
  Berlin, 1986.

\bibitem{Sohr}
H.~Sohr.
\newblock {\em The {N}avier-{S}tokes {E}quations. {A}n {E}lementary
  {F}unctional {A}nalytic {A}pproach}.
\newblock Birkh\"auser, Basel, 2001.

\bibitem{Triebel1978}
H.~Triebel.
\newblock {\em Interpolation Theory, Function Spaces, Differential Operators}.
\newblock Deutscher Verlag der Wissenschaften, Berlin, 1978.

\bibitem{Weis-Girardi}
L.~Weis and M.~Girardi.
\newblock Criteria for {R}-boundedness of operator families.
\newblock In {\em Evolution equations}, pages 203--221. Dekker, New York, 2003.

\bibitem{Xiang-Yan}
Z.~Xiang and W.~Yan.
\newblock On the well-posedness of the quasi-geostrophic equation in the
  {T}riebel-{L}izorkin-{L}orentz spaces.
\newblock {\em J. Evol. Equ.}, 11(2):241--263, 2011.

\bibitem{Yang-Cheng-Peng}
Q.~Yang, Z.~Cheng, and L.~Peng.
\newblock Uniform characterization of function spaces by wavelets.
\newblock {\em Acta Math. Sci. Ser. A Chin. Ed.}, 25(1):130--144, 2005.

\end{thebibliography}



\end{document}